\documentclass[12pt]{article}
\usepackage{amsmath, amsfonts, amsthm, amssymb, color, hyperref, extarrows, enumitem, nicefrac, blindtext, mathtools, comment, enumerate}

\newtheorem{theorem}{Theorem}[section]

\newtheorem{cor}[theorem]{Corollary}
\newtheorem{lem}[theorem]{Lemma}
\newtheorem{pro}[theorem]{Proposition}

\numberwithin{equation}{section}

\newtheorem{remark}[theorem]{Remark}

\newtheorem{example}{Example}

\usepackage[hyperpageref]{backref}

\usepackage{cite}

\hypersetup{
	colorlinks   = true,
	citecolor    = magenta}
	
\begin{document}
\title{\vspace{-1cm} \bf Unique continuation of Schr\"odinger-type equations for $\bar\partial$ \rm}
\author{Yifei Pan  \ \ and \ \  Yuan Zhang}
\date{}

\maketitle

\begin{abstract}
The purpose of this paper is to study the unique continuation property for  a   Schr\"odinger-type equation $ \bar\partial u = Vu$ on a domain in $\mathbb C^n$, where the solution $u$ may be a  scalar function, or a vector-valued function. While simple examples show that the unique continuation property  fails in general if  the potential $V\in L^{p}, p<2n$, we first prove that, in the case when  $u $ is a scalar function,    the unique continuation property holds when  $V\in L_{loc}^{2n}$ and is $\bar\partial$-closed.   For   vector-valued smooth solutions, we establish the unique continuation property  either when     $V\in L_{loc}^p $, $  p>2n$ for $n\ge 3$, or when $V\in L_{loc}^{2n}$ for $n =  2$. Finally,  we   discuss the unique continuation property for some special cases where   $V\notin L_{loc}^{2n}$, for instance, $V $ is a constant multiple of $ \frac{1}{|z|}$.

\end{abstract}

\renewcommand{\thefootnote}{\fnsymbol{footnote}}
\footnotetext{\hspace*{-7mm}
\begin{tabular}{@{}r@{}p{16.5cm}@{}}
& 2010 Mathematics Subject Classification. Primary 32W05; Secondary 35J10,  35A02.\\
& Key words and phrases. unique continuation, $\bar\partial$ operator, Schr\"odinger-type, 
 Moser-Trudinger inequality.
\end{tabular}}

\section{Introduction}

Let $\Omega$ be a   domain in $\mathbb C^n, n\ge 1$. Let $u: \Omega\rightarrow \mathbb C^N$  be a $H^1_{loc}(\Omega) $ solution to  the following Schr\"odinger-type equation for the $\bar\partial$ operator: 
\begin{equation}\label{eqn}
    \bar\partial u = Vu\ \ \text{on} \ \ \Omega  
\end{equation}  
in the sense of distributions.  Here the  potential $V$ is an $N\times N$ matrix of $(0, 1)$ forms with $L^p_{loc}(\Omega)$ coefficients for some $p\ge 1$,  and the space $H^k_{loc}(\Omega): = W^{k, 2}_{loc}(\Omega)$, where $W^{k, p}_{loc}(\Omega) $ is the standard Sobolev space of functions whose weak derivatives up to order $k$ exist and belong to $L^p_{loc}(\Omega)$. The  equation \eqref{eqn} arises naturally from various questions  in CR and almost complex geometry and plays an important role, for instance, while studying the boundary regularity and  uniqueness of CR-mappings, as well as uniqueness of J-holomorphic curves.  See  \cite{BL90, GR} et al. 
 
 In this paper, we  study the (strong) unique continuation property of \eqref{eqn}. Namely, we investigate whether a solution to \eqref{eqn} vanishing to infinite order  in the $L^2$ sense at one point vanishes identically. Here a    function $u\in L^2_{loc}(\Omega)$ is said to vanish to infinite order (or, be flat) in the $L^2$ sense at a point $z_0\in \Omega$  if for all $m\ge 1$,
\begin{equation*}\label{flat}
    \lim_{r\rightarrow 0} r^{-m}\int_{|z-z_0|<r}|u(z)|^2 dv_z =0,
\end{equation*} 
where $dv_z$  is the Lebesgue  measure   in $\mathbb C^n$ with respect to the dummy  variable $z$. Otherwise, $u$ is said to vanish to a finite order in the $L^2$ sense at $z_0$.

 As demonstrated by Example \ref{pf}, the unique continuation property    fails in general for \eqref{eqn} with  $L_{loc}^{p}$ potentials, $p<2n =dim_{\mathbb R} \Omega, $ the real dimension of the source domain $\Omega$. On the other hand, it should be reminded that for the real Laplacian $\Delta$, the unique continuation property   has been thoroughly understood.  In particular,   the works of Chanillo-Sawyer \cite{CS90} and Wolff  \cite{Wo90, Wo94} have shown that for a domain $\Omega\subset \mathbb R^d$ and  $V\in L_{loc}^{d}(\Omega)$, the unique continuation property for $H^2_{loc}(\Omega)$ solutions  of the differential inequality   \begin{equation} \label{lap}
    |\Delta u|\leq V|\nabla u|\ \ \text{on}\ \ \Omega 
\end{equation}  
  holds when $d=2, 3, 4$, and  fails in general when $d\ge 5.$

  Surprisingly, due to  the   more rigid structure of    $\bar\partial$,  the   unique continuation property of \eqref{eqn} holds for all $\bar\partial$-closed  $ L_{loc}^{2n}(\Omega)$ potentials,  $n (=dim_{\mathbb C}\Omega) \ge 1$,  as stated in the following theorem in the case when their solutions are scalar functions. This dimension independence of  the unique continuation property for $\bar\partial$ stands in stark contrast to the aforementioned result  for  $\Delta$.   In view of Example \ref{pf}, it is also optimal. 

\begin{theorem}\label{main2}
    Let $\Omega$ be a  domain in $\mathbb C^n$. Suppose $u: \Omega\rightarrow \mathbb C$ with $u\in H^1_{loc}(\Omega)$,  and satisfies $ \bar\partial u = V u$ on $\Omega$ in the sense of distributions
 for some  $\bar\partial$-closed  $(0,1)$ form $V\in L_{loc}^{2n}(\Omega)$.  If $u$ vanishes to infinite order in the $L^2$ sense at some $z_0\in \Omega$, then $u$ vanishes identically.
\end{theorem}

The $n=1$ case of the theorem was established in \cite{PZ} (for arbitrary target dimension $N$, see also Theorem \ref{pz}); the real-valued solution case has been proved lately in \cite{CPZ} concerning the gradient operator $\nabla$, given the equivalence of $\bar\partial$ to $\nabla$ on  such solutions. See also Corollary \ref{nod} for a similar result for smooth functions satisfying the inequality $ |\bar\partial u| \le V|u|$ for $V\in L^{2n}_{loc}$.  The proof of Theorem \ref{main2} 
relies on a   classification result of weak solutions to \eqref{eqn} below.

\begin{theorem}\label{main}
Let $\Omega$ be a pseudoconvex domain in $\mathbb C^n$. Given a $\bar\partial$-closed $(0,1)$ form $V\in L_{loc}^{2n}(\Omega)$, there exists a   function  $ f \in W^{1,2n}_{loc}(\Omega)$ such that every  $ H^1_{loc}(\Omega)$ solution $u: \Omega\rightarrow \mathbb C$   to $  \bar\partial u = Vu$ on $\Omega$ in the sense of distributions is of the form $e^{{f}}h $, for some holomorphic function $h$ on $\Omega$. In particular, $u\in W^{1,q}_{loc}(\Omega)$   for all $1\le q<2n$.
\end{theorem}

In the second part of the paper, we study the case when solutions to \eqref{eqn}, or to the following general inequality, are vector-valued (i.e.,  the target dimension $N\ge 1$): 
\begin{equation}\label{eqn2}
     |\bar\partial u| \le V|u|\ \ \text{a.e. on}\ \ \Omega.
\end{equation}
Here the potential $V$ is a nonnegative  scalar function in $L_{loc}^p(\Omega)$ for some $p\ge 1$. With the help of a complex polar coordinate formula in Lemma \ref{cp}, we convert the unique continuation problem on a source domain of dimension $n$ to that on the complex plane, where \cite{PZ} can  readily take into effect. As a consequence of this, we  prove in Section \ref{s4} that for smooth solutions of \eqref{eqn2}, the strong unique continuation property  holds for   $L_{loc}^{p}$ potentials, $p>2n$. Note that in  smooth category, a function vanishes to infinite order in the $L^2$ sense at a point if and only if it vanishes to infinite order in the usual jet sense at that point, that is, all  its derivatives vanish  at that point, see Lemma \ref{fn}.

\begin{theorem}\label{main3}
    Let $\Omega$ be a  domain in $\mathbb C^n$. Suppose $u: \Omega\rightarrow \mathbb C^N$ with $u\in C^\infty(\Omega)$,  and satisfies $ |\bar\partial u| \le V|u|$ a.e. on $\Omega$ for  $V \in L_{loc}^{p}(\Omega)$, $p>2n$.  If $u$ vanishes to infinite order   at some $z_0\in \Omega$, then $u$ vanishes identically.
\end{theorem}

Specifically, in the case  when $n=2$, we prove the unique continuation property of \eqref{eqn2} for $L_{loc}^{4}$ potentials, which, as indicated by Example \ref{pf}, is sharp. The key to  its proof in Section \ref{se6} incorporates   a weighted estimate of the Cauchy integral   established in \cite{PZ} and  a Carleman inequality Proposition \ref{hh3} for $\bar\partial$.

 \begin{theorem}\label{main6}
    Let $\Omega$ be a  domain in $\mathbb C^2$. Suppose $u: \Omega\rightarrow \mathbb C^N$ with $u\in C^\infty(\Omega)$,  and satisfies $ |\bar\partial u| \le V|u|$ a.e. on $\Omega$ for some    $V \in L_{loc}^{4}(\Omega)$.  If $u$ vanishes to infinite order   at some $z_0\in \Omega$, then $u$ vanishes identically.
\end{theorem}

   Due to Theorem \ref{main2} and Example \ref{pf}, a natural question arises about  whether the strong unique continuation property holds for \eqref{eqn2} with $L_{loc}^{2n}$ potentials in the vector-valued solutions case for any complex source  dimension $n$.     At this point we are only able to establish   Theorem \ref{main6} for  $n=2$ (and in \cite{PZ} for $n=1$). It remains unclear whether this property continues to be true when $n\ge 3$, in particular, in view of Wolff's  intricate counter-examples  to \eqref{lap} in higher dimensional cases (with the real source dimension $d\ge 5$). See Remark \ref{re} for  unsolved  questions along this line in detail. However, it is noteworthy that   the weak unique continuation property holds even for   $  L^2_{loc} $ potentials, as shown in \cite{PZ}. Namely, any solution to \eqref{eqn2}   vanishing on an open subset must vanish identically.

Finally, despite the general failure of the unique continuation property for \eqref{eqn2} with $ L^p_{loc}$ potentials, $ p<2n$, we explore in Section \ref{se5} and Section \ref{se7} a special case where $V\notin L^{2n}_{loc}$, yet the unique continuation   property may still be  anticipated. More precisely, $V$ here takes the form of   a constant multiple of $\frac{1}{|z|}$. Interestingly,  the cases of $N=1$ and $N\ge 2$  under this context are rather distinct: the unique continuation property  holds true  for all positive constant multiple $C$ when $N=1$, while when   $N\ge 2$, this   property fails in general if $C$ is  large, see Example \ref{ex2}.  

\begin{theorem}\label{main4}
    Let $\Omega$ be a  domain in $\mathbb C^n$ and $0\in \Omega$. Let $u: \Omega\rightarrow \mathbb C^N$ with $u\in C^\infty(\Omega)$,  and  satisfy $ |\bar\partial u| \le \frac{C}{|z|}|u|$ a.e. on $\Omega$. Assume $u$  vanishes to infinite order  at $0\in \Omega$.\\
    1). If $N=1$, then $u $ vanishes identically.  \\
    2). If $N\ge 2$ and $C< \frac{1}{4 }$, then $u$ vanishes identically.
\end{theorem}

We point out that in the case when either $N=1$ or $n=1$,   the smoothness assumption on $u$ above can be relaxed to $u\in H_{loc}^1(\Omega)$, as established in Theorem \ref{33} and Theorem \ref{mm1}. See also Theorem \ref{333} for the unique continue when the potentials include both powers of $\frac{1}{|z|}$ and Lebesgue integrable functions.  As an application, it allows us to  refine an earlier result in \cite{CPZ} in terms of $\nabla$, which states that  near any flat point of a smooth function $u$, either  $ \frac{|\nabla u|}{|u|}\notin L^{2n}$,  or $u$ vanishes identically there.  More precisely, denote by $u^{-1}(0)$ the zero set of a smooth  function $u$. We obtain in  Section \ref{se5}  the following blowing-up property   in terms of $\bar\partial$   near a flat point of $u$.

\begin{cor}\label{prn}
    Let $\Omega$ be a domain in $\mathbb C^n$. Suppose $u: \Omega \rightarrow \mathbb C$ with $u\in C^\infty(\Omega)$, and  vanishes to infinite order  at some  $z_0\in \Omega$. Then  for every neighborhood $U$ of $z_0$ in $\Omega$, either $ U\setminus u^{-1}(0) =\emptyset$, or 
        \begin{equation}\label{eq4}
                \int_{U\setminus u^{-1}(0) }\frac{|\bar\partial u |^{2n}}{|u |^{2n}}\ dv =\infty.
        \end{equation}
         \end{cor}
     \medskip
     
 \begin{remark}
       One can  compare   Corollary \ref{prn} with the following entertaining facts for compactly supported functions on real and complex Euclidean spaces. \\
       \noindent \textbf{1.} \cite[Theorem 2.7]{CPZ} For any $u: \mathbb R^d\rightarrow \mathbb C$ with $u \in C_c^\infty(\mathbb R^d), d\ge 2$,
$$ \int_{  \text{supp}\  u }\frac{|\nabla u |^{2}}{|u |^{2}}\ dv =\infty.  $$
  \noindent \textbf{2.} \cite[Theorem 1.3]{PZ}  For any $u: \mathbb C^n\rightarrow \mathbb C$ with $u\in C_c^\infty(\mathbb C^n)$, 
\begin{equation}\label{w2}
    \int_{  \text{supp}\  u }\frac{|\bar\partial u |^{2}}{|u |^{2}}\ dv =\infty. 
\end{equation}
The power 2 in \eqref{w2} is optimal, in view of an example $u_0\in  C_c^\infty(\mathbb C)$ in \cite{Ma02} by Mandache, which satisfies  for all $p<2$, 
  $$ \int_{  \text{supp}\  u_0 }\frac{|\bar\partial u_0 |^{p}}{|u_0 |^{p}}\ dv <\infty.$$ 
 \end{remark}
   \bigskip
   
 \noindent\textbf{Acknowledgments: }
  Part of this work by the first  author was conducted while he was on
  sabbatical leave visiting Huaqiao university in China in Spring 2024.  He thanks Jianfei Wang for  invitation, and 
  the host institution for   hospitality and excellent research
  environment.

\section{Moser-Trudinger inequality and applications}
Let $\Omega$ be a bounded domain in $\mathbb R^d$.  One of the technical aspects to prove our main theorems is a chain rule for  weak derivatives of  the exponential of $ W^{1, d}$ functions. In this section, we shall show

\begin{pro}\label{pde}
Let $\Omega$ be a  domain in $\mathbb R^d$ and  $f\in W_{loc}^{1, d}(\Omega)$. Then $e^f\in W_{loc}^{1, q}(\Omega)$ for all $1\le q<d$. Moreover, $\nabla e^f = e^f\nabla f$ in the sense of distributions. 
\end{pro}

The  $W^{1,d}$ space is the critical Sobolev space where the Sobolev embedding theorem fails, and instead is substituted by the classical \noindent\textbf{Moser-Trudinger inequality}.  Recall that the Moser-Trudinger inequality  states (see \cite{Mo}) that for   a bounded domain $\Omega\subset \mathbb R^d$ with Lipschitz boundary, there exists a positive constant  $C_{MT}$ depending only on $d$ such that
$$ \sup_{u\in W_0^{1, d}(\Omega),\ \  \|\nabla u\|_{L^d(\Omega)}\le 1 }\int_\Omega e^{\alpha_d |u|^{\frac{d}{d-1}}}dv \le C_{MT} |\Omega| .$$
Here  $\alpha_d: = dw_{d-1}^\frac{1}{d-1}$, with $w_{d-1}$  the surface area of the unit sphere in $\mathbb R^d$, and $ |\Omega|$  the volume of $\Omega$.   It turns out that the Moser-Trudinger inequality is exactly the key to prove Proposition \ref{pde}. Before proceeding to its proof, we first make use of the inequality to show that the exponential of $ W^{1, d}$ functions belongs to $L^p$ for all $p<\infty$. 

\begin{lem}\label{MTC}
Let $\Omega$ be a bounded Lipschitz domain in $\mathbb R^d$ and  $f\in W^{1, d}(\Omega)$. Then  for any $1\le p< \infty$, $e^{|f|}\in L^p(\Omega) $ with
\begin{equation}\label{MT}
    \left\|e^{|f|}\right\|^p_{L^p(\Omega)}  \le   2|\Omega|\left(  e^{\frac{C_\Omega p^{d}\|f\|_{W^{1,d}(\Omega)}^d}{\alpha^{d-1}_d}}+   C_{MT}\right),
\end{equation}  
for some constant $C_\Omega$ dependent only on $\Omega$. In particular, $e^f\in L^p(\Omega)$.
Equivalently, if  $\log |g|\in W^{1, d}(\Omega)$ for some function $g$ on $\Omega$, then $g, \frac{1}{g}\in L^p(\Omega)$ for all $1\le p< \infty$.
\end{lem}

\begin{proof}
Extend $f$ to be a function $\tilde f$ on  a bounded Lipschitz  domain $\tilde \Omega$,   such that $  \Omega\subset \subset \tilde \Omega,  |\tilde \Omega|\le 2 |\Omega|, $ $\tilde f\in W^{1, d}_0(\tilde \Omega)$ with   \begin{equation}\label{ext}
     \|\nabla \tilde f\|^d_{  L^d(\tilde \Omega)}\le \| \tilde  f\|^d_{  W^{1, d}(\tilde \Omega)}\le C_\Omega\|  f\|^d_{  W^{1, d}(\Omega)},
\end{equation} 
with $C_\Omega$ dependent only on $\Omega$. See \cite[pp. 268]{Ev}. Then
\begin{equation*}
    \begin{split}
        \int_\Omega e^{p|f|}dv\le \int_{\tilde \Omega} e^{p|\tilde f|}dv
        &= \int_{x\in \tilde \Omega, |\tilde f(x)|\le \frac{p^{d-1}\|\nabla \tilde f\|_{L^d(\tilde \Omega)}^d}{\alpha^{d-1}_d}}e^{p|\tilde f|}dv + \int_{x\in \tilde \Omega, |\tilde f(x)|> \frac{p^{d-1}\|\nabla\tilde f\|_{L^d(\tilde \Omega)}^d}{\alpha^{d-1}_d}} e^{p|\tilde f|}dv.
        \end{split}
        \end{equation*}
Since $$ \left\{x\in \tilde \Omega: |\tilde f(x)|> \frac{p^{d-1}\|\nabla\tilde f\|_{L^d(\tilde \Omega)}^d}{\alpha^{d-1}_d}\right\}  =   \left\{x\in \tilde \Omega:  p|\tilde f(x)|\le \frac{\alpha_d |\tilde f(x)|^\frac{d}{d-1}} {\|\nabla \tilde f\|_{L^d(\tilde \Omega)}^\frac{d}{d-1}}\right\},$$
we further have 
        \begin{equation*}
    \begin{split}
      \int_\Omega e^{p|f|} dv &\le e^{\frac{p^{d}\|\nabla \tilde f\|_{L^d(\tilde \Omega)}^d}{\alpha^{d-1}_d}}|\tilde \Omega| + \int_{x\in \tilde \Omega, p|\tilde f(x)|\le \frac{\alpha_d |\tilde f(x)|^\frac{d}{d-1}}{\|\nabla \tilde f\|_{L^d(\tilde \Omega)}^\frac{d}{d-1}}} e^{p|\tilde f|} dv
         \le e^{\frac{p^{d}\|\nabla \tilde f\|_{L^d(\tilde \Omega)}^d}{\alpha^{d-1}_d}}|\tilde \Omega| + \int_{ \tilde \Omega} e^{\alpha_d |f_1|^{\frac{d}{d-1}}}dv,
    \end{split}
\end{equation*}
where $f_1: =  \frac{ \tilde f}{\|\nabla \tilde f\|_{L^d(\tilde \Omega)}}$. Note that $f_1 \in W_0^{1, d}(\tilde \Omega)$ and $\|\nabla f_1\|_{L^d(\tilde \Omega)}= 1$. Applying the Moser-Trudinger inequality to $f_1$ in the last inequality and making use of \eqref{ext}, we get
$$ \int_\Omega e^{p|f|} dv\le |\tilde \Omega|\left(  e^{\frac{p^{d}\|\nabla \tilde f\|_{L^d(\tilde \Omega)}^d}{\alpha^{d-1}_d}}+   C_{MT}\right) \le   2|\Omega|\left(  e^{\frac{C_\Omega p^{d}\|f\|_{W^{1,d}(\Omega)}^d}{\alpha^{d-1}_d}}+   C_{MT}\right). $$
\eqref{MT} is proved. Since $|e^f|\le e^{|f|}, $ we further have $e^f\in L^p(\Omega)$. That $g, \frac{1}{g}\in L^p(\Omega)$ follows immediately from the facts that $|g| = e^{f}$ and $\frac{1}{|g|} = e^{-f}$ with $f: =\log|g|\in W^{1,d}(\Omega)$. 
\end{proof}
\medskip

It is worthwhile to note  that the integrability assumption $f\in W^{1, d}(\Omega)$ in Lemma \ref{MTC} is optimal in view of the following example. Denote by $B_r$  the ball in $\mathbb R^d$ centered at $0$ with radius $r$. 
\begin{example}For each $k\in \mathbb N$, let
$$f = -\ln|x|^{2k}, x\in B_{\frac{1}{2}}\subset \mathbb R^d, d\ge 2.$$
A direct computation shows that for each $1\le p<\infty$, $f\in L^p(B_{\frac{1}{2}})$  and 
 $ \nabla f = -\frac{2kx}{|x|^2}$ on  $B_{\frac{1}{2}}\setminus\{0\}.$
By a result of Harvey-Polking in \cite{HP}, we have $ \nabla f = -\frac{2kx}{|x|^2}$ on $B_{\frac{1}{2}}$ in the sense of distributions. Consequently,    $\nabla f\in L^q(B_{\frac{1}{2}})$ for all $ q<d$, and thus $$f\in W^{1, q}(B_{\frac{1}{2}})\ \ \text{for all}\ \ q<d.$$ 
On the other hand, since $ e^f =\frac{1}{|x|^{2k}} $ on $B_{\frac{1}{2}} $, 
$$e^f   \notin L^1(B_{\frac{1}{2}}) \ \ \text{if}\ \ k\ge \frac{d}{2}.$$
\end{example}
\medskip

\begin{proof}[Proof of Proposition \ref{pde}:]
Firstly, according to Lemma \ref{MTC}, $e^f\in L_{loc}^p(\Omega)$ for all  $1\le p<\infty$. By H\"older's inequality,  we have $e^f\nabla f\in L_{loc}^{q}(\Omega)$ for each $1\le q<d$, and for every Lipschitz subdomain $\tilde \Omega\subset\subset \Omega$, 
\begin{equation}\label{qq1}
\|e^f\nabla f\|_{L^q(\tilde \Omega)} \le  \|\nabla f\|_{L^d(\tilde\Omega)}\|e^{f}\|_{L^{q^*}(\tilde \Omega)}<\infty,
\end{equation}
where  $q^*: = \frac{dq}{d-q}$. We next show that $\nabla e^f = e^f\nabla f$ in the sense of distributions. If so, then $e^f\in W_{loc}^{1, q}(\Omega)$ by \eqref{qq1} for all $1\le q<d$, completing the proof.  

$\nabla e^f = e^f\nabla f $ is trivially true if $f\in C^\infty(\Omega)$ as a consequence of the classical chain rule. 
For general $f\in W_{loc}^{1, d}(\Omega)$ and any Lipschitz subdomain $\tilde \Omega\subset\subset \Omega$, let $f_k\in  C^\infty({\tilde \Omega})$ converge  to $ f$ in the $W^{1, d}({\tilde \Omega}) $ norm. By Sobolev embedding theorem, for all $1\le p<\infty$, \begin{equation}\label{e1}
    \|f_k-f\|_{L^p({\tilde \Omega})}\rightarrow 0 
\end{equation} as $k\rightarrow \infty$.  Moreover, applying  Lemma \ref{MTC} to $f$ and $f_k$, we have $e^f, e^{f_k}\in L^p({\tilde \Omega})$ for all $1\le p<\infty$, with    \begin{equation}\label{e2}
    \|e^f\|_{L^p({\tilde \Omega})} + \|e^{f_k}\|_{L^p({\tilde \Omega})}\le C
\end{equation}  
for some constant $C$ dependent only on $\|f\|_{W^{1, d}({\tilde \Omega})}$, ${\tilde \Omega}$ and $p$.

We claim that  $ e^{f_k}\rightarrow e^f$ in the $L^p$ norm, $1\le p<\infty$,  and $ \nabla e^{f_k}\rightarrow e^f\nabla f$ in the $L^q$ norm, $1\le q<d$. We shall need  the following elementary inequality as a consequence of the  mean-value theorem: for $z_1, z_2\in \mathbb C,$
$$ |e^{z_1} -e^{z_2}| \le \sup_{t\in[0,1]}\left|e^{tz_1+(1-t)z_2}\right||z_1- z_2| \le \left|e^{|z_1|}+ e^{|z_2|}\right||z_1- z_2|. $$
Making use of this inequality,  H\"older's inequality and \eqref{e1}-\eqref{e2}, we have  for every $1\le p<\infty$,
\begin{equation}\label{e3}
    \begin{split}
       \|e^{f_k} -e^f\|_{L^p({\tilde \Omega})} \le&   \| (e^{|f_k|} +e^{|f|}) |f_k-f|\|_{L^p({\tilde \Omega})} \le \|e^{|f_k|} +e^{|f|}\|_{L^{2p}({\tilde \Omega})}\|f_k-f\|_{L^{2p}({\tilde \Omega})} \rightarrow 0 
    \end{split}
\end{equation}
 as $k\rightarrow \infty$. Moreover, for each $q<d$, noting that  $\nabla e^{f_k} = e^{f_k}\nabla f_k$,   by \eqref{e2} and \eqref{e3}
\begin{equation*}
    \begin{split}
       \|\nabla e^{f_k} - e^f\nabla f\|_{L^q({\tilde \Omega})} \le&    \|(e^{f_k} - e^f)\nabla f_k\|_{L^q({\tilde \Omega})} +  \| e^f(\nabla f_k-\nabla f)\|_{L^q({\tilde \Omega})} \\
        \le&   \|e^{f_k} - e^f\|_{L^{q^*}({\tilde \Omega})}\|\nabla f_k\|_{L^d({\tilde \Omega})} +  \| e^f\|_{L^{q^*}({\tilde \Omega})}\|\nabla f_k-\nabla f\|_{L^d({\tilde \Omega})} \rightarrow 0 
    \end{split}
\end{equation*}
 as $k\rightarrow \infty$.    The claim is proved. In particular, it immediately gives  $ e^{f_k}\rightarrow e^f$ and $ \nabla e^{f_k}\rightarrow e^f\nabla f$  on $\tilde \Omega$ in the   sense of distributions, and thus   $\nabla e^f = e^f\nabla f$ on $\tilde \Omega$ in the sense of distributions. 
\end{proof}

\medskip

At the end of the section, we discuss another immediate application of Lemma \ref{MTC}. We say a function $f$ to be  H\"older at a point $x_0$ if there exists some $\alpha\in (0, 1]$ and a constant $C>0$ such that for all $x$ near $x_0$, 
$$|f(x)-f(x_0)|\le C|x-x_0|^\alpha.$$  The following corollary       states that  the  logarithms of such functions are never in $W^{1,d}$ near $x_0$. This also generalizes a similar result in \cite{CPZ} for Lipschitz functions (i.e., $\alpha =1$). 

\begin{cor}
Let $f$ be a function   near $x_0$ in $\mathbb R^d$ and be H\"older at $x_0$.  Then $\ln|f(x)-f(x_0)|\notin W^{1,d}$ near $x_0$.
\end{cor}

\begin{proof}
Supposing not. Then by Lemma \ref{MTC}, for all $1\le p<\infty$,
$$\frac{1}{|f(x)-f(x_0)|} = e^{- \ln|f(x)-f(x_0)|} \in L^p $$
near $x_0$. However, by the H\"older  property of $f$ at $x_0$, this would imply that there exist  some constants $0<\alpha<1$ and   $C>0$, such that
$$\frac{1}{|f(x)-f(x_0)|} \ge \frac{1}{C|x-x_0|^\alpha}\in L^p  $$
near $x_0$,  which is absurd when $p\ge \frac{d}{\alpha}$.
\end{proof}

\section{Unique continuation for the target dimension $N=1$}\label{se3}
In this section, we prove  the classification Theorem \ref{main} of weak solutions to $\bar\partial$,   and the unique continuation   Theorem \ref{main2} for scalar  solutions ($N=1$) in a domain $\Omega\subset \mathbb C^n, n\ge 1$.   Let us first point out that, given any solution $u$ to \eqref{eqn},  a formal computation leads to 
$$0= \bar\partial^2 u = u\bar\partial V - V\wedge \bar\partial u = u\bar\partial V - u V\wedge V = u\bar\partial V\ \ \text{on}\ \ \Omega. $$
In this sense, it is   natural to assume $V$ to be $\bar\partial$-closed  in Theorem \ref{main} and Theorem \ref{main2}. 

The following lemma concerning  the local ellipticity of $\bar\partial$ for $(0,1)$ data with $W_{loc}^{k, p}$  coefficients  is well-known for $p=2$ (see, for instance, \cite[Theorem 4.5.1]{CS}). However, it seems difficult to find    a reference for general $p, 1<p<\infty$. Since this property will be repeatedly used in the paper, we present a proof below. 

\begin{lem}\label{el}
    Let $\Omega$ be a   domain in $\mathbb C^n$, and $1<p<\infty$. Let $V \in L_{loc}^{p}(\Omega)$ be a $\bar\partial$-closed $(0,1)$ form on $\Omega$. Then every  solution to $\bar\partial f =V $ on $\Omega$ in the sense of distributions belongs to $ W_{loc}^{1, p}(\Omega) $. Furthermore,  if $V \in W_{loc}^{k, p}(\Omega)$,  $k\in \mathbb Z^+$, then every  solution to $\bar\partial f =V $ on $\Omega$ in the sense of distributions belongs to $ W_{loc}^{k+1, p}(\Omega) $.
\end{lem}

\begin{proof}
Suppose the $\bar\partial$-closed $(0,1 )$ form $V$ belongs to $ W_{loc}^{k, p}(\Omega)$ for some $k\in \mathbb Z^+\cup\{0\}$. Since the lemma is  purely local, and every other solution is differed only by a holomorphic function, it suffices to show that for any $z_0\in \Omega$, there exist a neighborhood $U$ of $z_0$ and a solution $f_0$ to $\bar\partial f =V $ on $U$ in the sense of distributions, such that   $f_0\in W_{loc}^{k+1, p}(U)$. For simplicity, let $z_0 =0$ and   $B_{2r}\subset\subset  \Omega$ for some $r>0$.  Let $\eta$ be a compactly supported function   on $B_{2r}$ such that $\eta=1$ on $B_{r}$.   

 Given a  mollifier $\phi$    on $\mathbb C^n$, we have $V_\epsilon: =V*\phi_\epsilon\in C^\infty(B_{2r})$, $V_\epsilon $ is $\bar\partial$-closed on $B_{2r}$ and $V_\epsilon\rightarrow V$ in the $L^p(B_{2r})$ norm. 
 Applying the Bochner-Martinelli representation formula   to $\eta V_\epsilon$ on $B_{2r}$, one has
 \begin{equation*} 
     \eta(z)V_\epsilon(z) = -\int_{B_{2r}}\bar\partial(\eta(\zeta)V_\epsilon(\zeta))\wedge B_1(\zeta, z) - \bar\partial \int_{B_{2r}}\eta(\zeta)V_\epsilon(\zeta)\wedge B_0(\zeta, z),  \ \   z\in B_{2r},
 \end{equation*}
 where for $q=0, 1$, $$B_q(\zeta, z) = -*\partial_\zeta \overline{\Gamma_q(\zeta, z)}$$ with 
 $$\Gamma_q(\zeta, z) = \frac{ (n-2)!}{q!2^{q+1}\pi^n} \frac{1}{|\zeta-z|^{2n-2}}\left(\sum_{j=1}^nd\bar\zeta_jdz_j\right)^q. $$
See, for instance, \cite[Chapter I]{LM}.  Note that $\bar\partial (\eta V_\epsilon) = \bar\partial \eta \wedge V_\epsilon $ on $B_{2r}$ and $\text{supp} \ \bar\partial \eta\subset  B_{2r}\setminus B_r$. Then   restricting on $B_r$, 
\begin{equation}\label{aaa}
     V_\epsilon(z) = -\int_{B_{2r}\setminus B_{r}}\bar\partial\eta(\zeta) \wedge V_\epsilon(\zeta) \wedge B_1(\zeta, z) - \bar\partial \int_{B_{2r}}\eta(\zeta)V_\epsilon(\zeta)\wedge B_0(\zeta, z),  \ \   z\in B_{r}. 
\end{equation} 
By  Young's convolution inequality, there exists some $C>0$ such that 
$$ \left \|\int_{B_{2r}\setminus B_{r}}\bar\partial\eta(\zeta) \wedge V_\epsilon(\zeta) \wedge B_1(\zeta, z) - \int_{B_{2r}\setminus B_{r}}\bar\partial\eta(\zeta) \wedge V(\zeta) \wedge B_1(\zeta, z)\right \|_{L^1(B_r)}\le C\|V_\epsilon -V\|_{L^1(B_{2r})},$$
which goes to $0$ as $\epsilon\rightarrow 0$. Similarly,
$$ \left \|\int_{B_{2r} } \eta(\zeta) \wedge V_\epsilon(\zeta) \wedge B_0(\zeta, z) - \int_{B_{2r} } \eta(\zeta) \wedge V(\zeta) \wedge B_0(\zeta, z)\right \|_{L^1(B_r)} \rightarrow 0  $$
as $\epsilon\rightarrow 0$. 
Thus passing $\epsilon\rightarrow 0$ in \eqref{aaa}, we obtain
\begin{equation}\label{bm1}
   V (z) = -\int_{B_{2r}\setminus B_{r}} \bar\partial\eta(\zeta) \wedge V(\zeta)\wedge B_1(\zeta, z) - \bar\partial \int_{B_{2r}}\eta(\zeta)V (\zeta)\wedge B_0(\zeta, z) \ \ \text{on}\ \ B_r
\end{equation}
  in the sense of distributions. 

Note that
$$ I: = - \int_{B_{2r}\setminus B_{r}} \bar\partial\eta(\zeta)\wedge V(\zeta) \wedge B_1(\zeta, z)\in C^\infty(B_r), $$
and   $I $ is $\bar\partial$-closed  on $B_r$ by \eqref{bm1}. By ellipticity of $\bar\partial$ for smooth data (see \cite[Theorem 4.5.1]{CS}), there exists a function $v_0\in C^\infty(B_r )$ such that $\bar\partial v_0 = I$ on $B_r $. 
On the other hand, according to the classical potential theory for the fundamental solution of Laplacian, 
$$  u_0: = -\int_{B_r}\eta(\zeta)V(\zeta)\wedge B_0(\zeta, z)\in W^{k+1, p}(B_r).$$
Letting \begin{equation}\label{bm}
    f_0:  = v_0+ u_0,
\end{equation}   we have $f_0\in W_{loc}^{k+1, p}(B_r)$,  and  $\bar\partial f_0 =V $ on $B_r $ in the sense of distributions by \eqref{bm1}.
\end{proof}
\medskip

\begin{proof}[Proof of Theorem \ref{main}: ] Since $V\in L_{loc}^2(\Omega)$ is $\bar\partial$-closed and $\Omega$ is pseudoconvex, by H\"ormander's $L^2$ theory (see \cite[Theorem 4.3.5]{CS}), there exists ${f}\in L_{loc}^2(\Omega)$  satisfying $\bar\partial{f} = V $ on $\Omega$. Noting that  $\bar\partial$ is an elliptic operator of order one for $(0,1)$ data by Lemma \ref{el} 
and $V\in L_{loc}^{2n}(\Omega) $, we further get ${f}\in W_{loc}^{1, 2n}(\Omega)$. 
Hence we can  apply  Proposition \ref{pde} to   $\tilde u: = e^{-{f}}$, and obtain that  $\tilde u \in W_{loc}^{1, q}(\Omega)$ for any $q< 2n$,   and  $\bar\partial \tilde u = -V \tilde u$ on $\Omega$ in the sense of distributions. In particular, $\tilde u\in L^p_{loc}(\Omega)$ for every $p<\infty$ by Sobolev embedding theorem.

 For each solution $u\in H_{loc}^{1}(\Omega)$ to $\bar\partial u = Vu$, consider $h: = u\tilde u$ on $\Omega$. Similarly as in  the proof of Proposition \ref{pde}, we  verify that $\bar\partial h= 0$ on $\Omega$ in the sense of distributions. In detail, let $u_k\in C^\infty(\Omega)\rightarrow u$ in the $H_{loc}^1$ norm. Then for every subdomain  $\tilde \Omega\subset\subset \Omega$, one has 
 \begin{equation}\label{e4}
     \| \bar\partial u_k - Vu\|_{L^2(\tilde \Omega)} = \| \bar\partial u_k - \bar\partial u\|_{L^2(\tilde \Omega)} \rightarrow 0  
 \end{equation} 
as $k\rightarrow \infty$.  Moreover, by   Sobolev embedding theorem
 \begin{equation}\label{e5}
     \|   u_k - u\|_{L^\frac{2n}{n-1}(\tilde \Omega)} \le \|   u_k -   u\|_{H^1(\tilde \Omega)}\rightarrow 0  
 \end{equation} as $k\rightarrow 0$. On the other hand, by H\"older's inequality, for each $k\ge 0$,  $h_k: = u_k\tilde u$ satisfies 
 $$\|h_k - h\|_{L^{1}(\tilde \Omega)} = \|(u_k - u)\tilde u\|_{L^{1}(\tilde \Omega)}\le \|u_k-u\|_{L^{2}(\tilde \Omega)} \|\tilde u\|_{L^{2}(\tilde \Omega)}\rightarrow 0$$ as $k\rightarrow \infty$. 
Since $u_k\in C^\infty(\Omega)$, the product rule applies to  give $\bar\partial h_k = \bar\partial u_k\tilde u  - Vu_k\tilde u $  on $  \Omega$ in the sense of distributions. Consequently, we use  H\"older's inequality again to infer
\begin{equation*}
\begin{split}
       \| \bar\partial h_k\|_{L^1(\tilde \Omega)} =&  \|\bar\partial u_k \tilde u - Vu_k \tilde u\|_{L^1(\tilde \Omega)} \le \|(\bar\partial u_k -Vu) \tilde u\|_{L^1(\tilde \Omega)}  +\| V(u_k-u) \tilde u \|_{L^1(\tilde \Omega)}\\
     \le&  \|\bar\partial u_k -Vu\|_{L^2(\tilde \Omega)}\| \tilde u\|_{L^2(\tilde \Omega)}  +\| V\|_{L^{2n}(\tilde \Omega) }\|u_k-u\|_{L^\frac{2n}{n-1}(\tilde \Omega)}\| \tilde u \|_{L^2(\tilde \Omega)}\rightarrow 0
\end{split}
\end{equation*}
 as $k\rightarrow \infty$ by \eqref{e4}-\eqref{e5}. In particular, $ h_k\rightarrow h$ and $\bar\partial h_k\rightarrow 0$ on $\Omega$  in the sense of distributions. Hence $\bar\partial h=0$ on $\Omega$ in the sense of distributions.  Altogether, $u = \tilde u^{-1}h = e^{f} h$ for some function $f\in W_{loc}^{2n}(\Omega)$, and some holomorphic function $h$ on $\Omega$. Moreover, since $e^f\in W^{1,q}_{loc}(\Omega)$   for all $1\le q<2n$ by Proposition \ref{pde}, so does $u$.
    \end{proof}
\medskip

Recall that  a smooth function vanishes to infinite order (or, is flat) in the jet sense at one point if all its derivatives vanish at that point. We verify below that for smooth functions, flatness in the jet sense and flatness in the $L^2$ sense are equivalent to each other.   

\begin{lem}\label{fn}
    Let $h$ be a smooth function near $x_0\in \mathbb R^d$. Then $h$ vanishes to infinite order in the $L^2$ sense at $x_0$ if and only if $h$ vanishes to infinite order in the  jet sense at $x_0$.  In particular, if $h$ is real-analytic near $x_0$, then either $h\equiv 0$ near $x_0$, or $h$   vanishes to a finite order in the $L^2$ sense at $x_0$.
\end{lem}

\begin{proof}
Without loss of generality  let $x_0=0$. If $h$ vanishes to infinite order in the  jet sense at $0$, then for any $m\ge 1 $, there exists a constant $C$ dependent on $m$ such that  $|h(x)|\le C|x|^m$ for  $|x|<<1$. Thus
$$ r^{-m}\int_{|x|<r} |h(x)|^2dv_x\le C  r^{-m}\int_{0}^r t^{d-1+2m} dt \le C r^{m+d}\rightarrow 0$$
 as $r$ goes to $0$. Namely, $h$ vanishes to infinite order in the $L^2$ sense at $0$.

Conversely, suppose  that  $h   $   vanishes to infinite order in the $L^2$ sense, but vanishes to a finite order in the  jet sense  at $0$. Let  $k>0$ be the smallest integer such that the $k$-th order homogeneous Taylor polynomial $p_k$ of $h$ is nonzero, and write   $h(x) = p_k(x)+ q_{k+1}(x)$, where   $q_{k+1}(x)$ is the  remaining term  of the Taylor expansion of $h$. By definition of $p_k$,     
    $$c_0: =  {r^{-2k}}\int_{|x|=1}|p_k(x)|^2dS_x>0 $$
and is  independent of $r$.  Let $r_0$ be such that for all $r<r_0$,
   $$ {r^{-2k}}\int_{|x|=1}|q_{k+1}(x)|^2dS_x\le \frac{c_0}{8}.  $$
    Then for any $0<r<r_0$, making use of the inequality  $|a+b|^2\ge \frac{3}{4}|a|^2- 3|b|^2$ for $a, b\in\mathbb R$, we have
    \begin{equation*}\begin{split}
                \int_{|x|<r} |h(x)|^2dv_x &= \int_{0}^r t^{d-1}\int_{|x|=1}  | p_k(x)+ q_{k+1}(x)|^2 dS_xdt \\
                &\ge  \int_{0}^r t^{d-1}\int_{|x|=1}  \frac{3}{4}| p_k(x)|^2- 3|q_{k+1}(x)|^2 dS_xdt\\
                &\ge \frac{3c_0}{8}\int_{0}^r t^{d-1}t^{2k}dt = \frac{3c_0}{8(d+2k)}r^{d+2k}.
    \end{split}
    \end{equation*}
    In particular,
    $$  \lim_{r\rightarrow 0} r^{-(d+2k) }\int_{|x|<r} |h(x)|^2dv_x\ge \frac{3c_0}{8(d+2k)}.$$
Contradiction! 

If $h$ is real-analytic near $0$, then either $h\equiv 0$ near $0$, or $h$ vanishes to a  finite order in the jet sense at $0$, which is further equivalent to vanishing  to a finite order in the $L^2$ sense at $0$. The proof is complete.\end{proof}
\medskip

\begin{proof}[Proof of Theorem \ref{main2}:]
 Again, let $z_0=0$, and $r_0>0$ be small  such that  $ B_{r_0} \subset \Omega$.  For each solution $u (\not\equiv 0)$ to $\bar\partial u = Vu$ on $B_{r_0} $ in the sense of distributions, by Theorem \ref{main}  there exists  a holomorphic function $h ( \not\equiv 0)$ on $B_{r_0}$ and a function $f\in W^{1, 2n}_{loc}(B_{r_0})$ such that    $u =e^f h $ on $B_{r_0}$. Applying Lemma \ref{MTC} to $-f$, we further have  $h = ue^{-f}$ with $e^{-f}\in L^2_{loc}(B_{r_0})$ in particular.  Consequently,  there exist  some constants $C_1, C_2>0$ such that  $$\sup_{|z|<\frac{r_0}{2}}|h|\le C_1 \ \ \text{and}\ \    \int_{ |z|<\frac{r_0}{2}}|e^{-f}|^2dv_z\le C_2.$$   
    Making use of  H\"older's inequality, we have
   for any $0<r<\frac{r_0}{2}$, 
\begin{equation*}
\begin{split}
    \left(\int_{|z|<r} |h|^2dv_z\right)^2 &\le C_1^2    \left(\int_{|z|<r} |h|dv_z\right)^2 
     \le C_1^2\int_{|z|<r} |u|^2dv_z \int_{|z|<r} |e^{-f}|^2dv_z \\
        & \le  C_1^2C_2   \int_{|z|<r} |u|^2dv_z.
\end{split}
\end{equation*}
  Since $h$ vanishes to a finite order in the $L^2$ sense at $0$ according to Lemma \ref{fn}, the same holds true for  $u$.  This completes the proof. 
\end{proof}
\medskip

The assumption  $V\in L^{2n}_{loc}$ in Theorem \ref{main2} can not be relaxed in the following sense. For each $p<2n$, there exists a differential equation $\bar\partial u = Vu$ with $V\in L^{p}_{loc}$,  and this equation has a nontrivial solution that vanishes to  infinite order  at a specific point. 

\begin{example}\label{pf}
For each $1\le p<2n$, let $\epsilon  \in (0, \frac{2n}{p}-1)$ and consider 
$$  \bar\partial u = \frac{\epsilon zd\bar z}{2 |z|^{\epsilon+2}}u\ \ \text{on}\ \ B_1\subset \mathbb C^n.$$
It is straightforward to verify that $ V:  =\frac{\epsilon zd\bar z}{2 |z|^{\epsilon+2}}\in L^p(B_1)$. On the other hand, $u_0= e^{-\frac{1}{|z|^\epsilon}}$ is a nontrivial solution to the above equation that  vanishes to infinite order in the $L^2$ sense at $0$. 

\end{example}

The proof to Theorem \ref{main2} indicates that solutions to \eqref{eqn} inherit the unique continuation property from that of holomorphic functions. However,    from the perspective of zero set, due to the presence of the other factor $e^f$, such solutions could exhibit a much   larger zero set than holomorphic functions do. In fact, the following example constructs a global potential  $V\in L^{2n}(\mathbb C^n)$, such that the ``zero set" of every weak solution to \eqref{eqn} with this potential contains a countable dense set in $\mathbb C^n$. It is noteworthy that the exponential factor $e^f$ no longer contributes zeros if $V\in L^p_{loc}, p>2n$, since in this case $f$   becomes continuous  by Sobolev embedding theorem.  

\begin{example}\label{ex5}
    Let $\phi(z) = -\ln(-\ln|z|)$ on $B_{\frac{1}{3}}\subset\mathbb C^n$,  and $\chi(\ge 0)\in C_c^\infty(B_{\frac{1}{3}})$ such that $\chi=1$ on $B_{\frac{1}{6}} $. Then $\psi: = \chi\phi\in W^{1, 2n}(\mathbb C^n)$ with  $\psi\le 0$ on $\mathbb C^n$. Given a countable dense  set $S: =\cup_{j=1}^\infty\{a_j\} \subset\mathbb C^n$, consider
    $$ f(z): = \sum_{j=1}^\infty 2^{-j}\psi(z-a_j), \ \ z\in \mathbb C^n. $$
    Clearly, $f\in W^{1, 2n}(\mathbb C^n)$. Further let $$V: = \bar\partial f\ \ \text{on} \ \ \mathbb C^n.$$ (Note that $\text{supp}\  V = \mathbb C^n$.) 
    According to the construction of $V$ and  Theorem \ref{main}, any solution $u$ to $\bar\partial u = Vu$ on $\mathbb C^n$ in the sense of distributions is of the from $e^f h$, for some holomorphic function $h$ on $\mathbb C^n$. Note that near every  $a_j\in S$, $$|e^f| \le \frac{1}{\left|\ln|z-a_j|\right|^{2^{-j}}}.$$  Hence   for all  $a_j\in S$,  $$\lim_{z\rightarrow a_j}u = 0. $$  
    On the other hand,    $u$ vanishes to a finite order in the $L^2$ sense  at  these points, as a consequence of Theorem \ref{main2}. 
\end{example}

As mentioned in the introduction, when the solutions are real-valued, Theorem \ref{main2} can be reduced to   a unique continuation property in \cite{CPZ} for the inequality $|\nabla  u|\le V|u|$ with $V\in L^{2n}_{loc}$, due to the equivalence of $\bar\partial$ and $\nabla$ for real-valued functions. The following example   constructed in \cite{GR} by Gong and Rosay 
carries a family of continuous solutions to   $\bar\partial u = Vu$ with some $V\in L^{2n}_{loc}$, yet  $\frac{|\nabla u|}{|u|}\notin L^{2n}_{loc}$, as a result of which \cite{CPZ} fails to apply. Instead,  we may use Theorem \ref{main2} to  conclude that none of the nontrivial solutions vanishes to infinite order at any point in the $L^2$ sense.

  \begin{example}\label{gr}
    Let $\{a_j\}_{j=1}^\infty $ be a sequence of distinct points in $B_\frac{1}{2}\subset \mathbb C^n$ convergent to $0$ and consider \begin{equation}\label{eq2}
        \bar\partial u =Vu\ \ \ \ \text{on}\ \ B_\frac{1}{2} 
    \end{equation}  with
    $$V: = \frac{zd\bar z}{|z|^2|\ln|z|^2|^2\Pi_{j=1}^\infty \left|\ln|z-a_j|^2\right|^{\frac{1}{j^2}}} +\sum_{k=1}^\infty\frac{(z-a_k)d\bar z}{ k^2|z-a_k|^2\ln|z|^2  \left|\ln|z-a_k|^2\right|^{\frac{1}{k^2}+1}\Pi_{j\ne k} \left|\ln|z-a_j|^2\right|^{\frac{1}{j^2}}}.$$  Then   $V\in L^{2n}(  B_\frac{1}{2})$.    \eqref{eq2} possesses a family of nontrivial solutions. In fact, for every holomorphic function $h$ on $B_\frac{1}{2} $,
    $$ u^h(z): = \frac{h(z)}{\ln|z|^2\Pi_{j=1}^\infty \left|\ln|z-a_j|^2\right|^{\frac{1}{j^2}}}  $$
 is continuous on $B_\frac{1}{2}$  and satisfies $\bar\partial u =Vu$ on $B_\frac{1}{2}\setminus \{\cup_{j=1}^\infty\{a_j\}\cup\{0\} \} $.  Applying a general removable singularity result in \cite{HP}, one further has  $\bar\partial u^h =Vu^h$ to hold on $ B_\frac{1}{2} $. 
 Note that the zero set $(u^h)^{-1}(0) = \cup_{j=1}^\infty\{a_j\}\cup h^{-1}(0)\cup\{0\} $. According to Theorem \ref{main2}, none of the nontrivial solutions to \eqref{eq2} vanishes to infinite order in the $L^2$ sense at any of these zero points.

In particular, if $h$ in the expression of $u^h$ is a holomorphic function of one variable with zeroes on $B_\frac{1}{2} $ (say, $h(z)=z_1$),  then    
$$|\nabla u^h| \approx |\bar\partial u^h| +|\partial u^h| \approx \left(V + \frac{|\partial h|}{|h|}\right)  |u^h|.$$ Since $\frac{\partial h}{h}\notin  L^2(B_\frac{1}{2})$ near any zero of $h$ (see  \cite[Proposition 5.2]{CPZ}), we have $$\frac{|\nabla u^h|}{|u^h|}\notin L^{2n}_{loc}(B_\frac{1}{2}),$$ where \cite{CPZ} fails to apply. 
\end{example}

One can compare Theorem \ref{main2} with a uniqueness result  for Lipschitz functions in \cite{CPZ}: if  $u$ is Lipschitz and satisfies $\nabla u =Vu$ for some $V\in L^{2n}(B_\frac{1}{2})$ and $u(0)=0$, then $u\equiv 0$. Note that if the holomorphic function $h$  in  Example \ref{gr} is nowhere zero on   $B_\frac{1}{2}$,  then the continuous function $u^h $ satisfies $\nabla u^h =Vu^h$ for some $V\in L_{loc}^{2n}(B_\frac{1}{2})$, and  has infinite many zeros on $B_\frac{1}{2}$. The uniqueness property fails for $u^h $   since it does not belong to $   Lip(B_\frac{1}{2})$.

\section{Unique continuation for $L^p$ potentials with $p\ge 2n$ }\label{s4}
In this section we prove Theorem \ref{main3} for smooth vector-valued solutions (where the target dimension $N\ge 1$) to  $ |\bar\partial u| \le V|u|$ a.e.  on a domain in $\mathbb C^n$  with $L^p$ potentials, $p\ge 2n$. When $N\ge 1$, this inequality  with a solution $u = (u_1, \ldots, u_N)$  reads as   \begin{equation*} 
     |\bar\partial u|: =\left(\sum_{j=1}^n\sum_{k=1}^N |\bar\partial_j u_k|^2\right)^{\frac{1}{2}}\le V \left(\sum_{k=1}^N|u_k|^2 \right)^{\frac{1}{2}}: = V|u|. 
\end{equation*}
The following unique continuation properties have been proved  for $L_{loc}^2$ potentials.   

\begin{theorem}\cite{PZ}\label{pz}
Let $\Omega$ be a  domain in $\mathbb C^n$. Suppose $u: \Omega\rightarrow \mathbb C^N$ with $u\in H^1_{loc}(\Omega)$,   and satisfies $ |\bar\partial u|\le V|u|$ a.e.\ on $\Omega$ for some $V\in L_{loc}^2(\Omega)$.\\ 
1). The weak unique continuation holds: if $u $ vanishes  in an open subset of $ \Omega$,   then $u$ vanishes identically.\\
 2). If $n=1$, then  the (strong) unique continuation holds: if $u$ vanishes to infinite order in the $L^2$ sense at some $z_0\in \Omega$, then $u$ vanishes identically.  
\end{theorem}

 In particular, since all the potentials under consideration in this paper    belong to $L_{loc}^2$ away from the reference point $z_0$, their unique continuation properties can be reduced to demonstrating that solutions vanish near a neighborhood of  $z_0$, in view of the above weak unique continuation property.

The  complex  polar  coordinates change formula below  will  play a crucial rule throughout the rest of the paper. One can find the formula that was  used  in \cite[pp. 260]{Ho} without proof. For the convenience of the reader, we provide its proof below.  Let $S^d$ be the unit sphere in $\mathbb R^{d+1}$, and 
$D_{r_0}$ be the   disk centered at $0$ with radius $r_0$ in $\mathbb C$. Recall that $B_{r_0}$ is the   ball centered at $0$ with radius $r_0$ in $\mathbb C^n$.

\begin{lem}\label{cp}
    Let $u\in L^1(B_{r_0})$. Then  for  a.e. $\zeta\in S^{2n-1}$,   $|w|^{2n-2}u(w\zeta)$  as a function of $w\in D_{r_0}$ is in $   L^1(D_{r_0})$, with 
    $$\int_{|z|< r_0} u(z) dv_z = \frac{1}{2\pi} \int_{|\zeta| =1}  \int_{|w|< r_0}|w|^{2n-2}u(w\zeta)dv_wdS_\zeta.$$
\end{lem}

\begin{proof}First, by the standard polar coordinates change, 
   \begin{equation*}
       \begin{split}
     2\pi \int_{|z|<r_0} u(z) dv_z =&       2\pi \int_0^{r_0}r^{2n-1}\int_{|\xi|=1}u(r\xi) dS_\xi  dr \\
     =&        \int_0^{r_0}r^{2n-1}\int_{|\eta|=1}\int_{|\xi|=1}u(r\xi) dS_\xi dS_\eta dr.       
       \end{split}
   \end{equation*}
  Here  $dS_\eta$ and $dS_\xi$   are the Lebesgue measures over $ S^1$ and $S^{2n-1}$, respectively. Since for each  $\eta\in S^1$,  
   $$ \int_{|\xi|=1}u(r\xi) dS_\xi  = \int_{|\zeta|=1}u(r\zeta\eta) |\eta|dS_\zeta = \int_{|\zeta|=1}u(r\zeta\eta)dS_\zeta, $$
   we further have by Fubini's theorem, 
   \begin{equation*}
       \begin{split}
      2\pi \int_{|z|<r_0} u(z) dv_z =&   \int_0^{r_0}r^{2n-1}\int_{|\eta|=1}\int_{|\zeta|=1}u(r\zeta\eta) dS_\zeta dS_\eta dr   \\
      =&\int_{|\zeta|=1}\int_0^{r_0}r^{2n-1}\int_{|\eta|=1}u(r\zeta\eta) dS_\eta drdS_\zeta\\
      =&\int_{|\zeta|=1}\int_0^{r_0}r\int_{|\eta|=1}|r\eta|^{2n-2}u(r\zeta\eta) dS_\eta drdS_\zeta\\
      =&\int_{|\zeta| =1}  \int_{|w|< r_0}|w|^{2n-2}u(w\zeta)dv_wdS_\zeta. 
       \end{split}
   \end{equation*}
\end{proof}

     As seen below, Lemma \ref{cp} allows us to transform \eqref{eqn2} with $L^p_{loc}, p>2n$ potentials into new ones with $L^2_{loc}$ potentials along almost all  complex one-dimensional radial directions. On the other hand, when the solutions under consideration are smooth, the  flatness of these solutions  at a point   naturally extends to   their  restrictions along those radial directions. Thus one can completely convert the unique continuation property   in the higher source dimension case to that on the complex one  dimension, where Theorem \ref{pz} can be applied.

\begin{proof}[Proof of Theorem \ref{main3}: ] Without loss of generality, let $z_0=0$ and $r>0$ be small such that  $V\in L^p(B_{r})$.  For each fixed $\zeta\in S^{2n-1}$, let $\tilde V(w): = |w|^{\frac{2n-2}{p}}V(w\zeta)$ and $v(w): = u(w\zeta), w\in D_{r}$. Since all jets of $u$ vanish at $0$ by assumption, the same holds true for all jets of $v$. Thus $v$ vanishes to infinite order at $0$ in the $L^2$ sense by Lemma \ref{fn}. Moreover, $v$ satisfies
$$|\bar\partial v(w) |= |\zeta \cdot \bar\partial u(w\zeta)|\le V(w\zeta) |u(w\zeta)| = |w|^{-\frac{2n-2}{p}}\tilde V(w)|v(w)|,\ \ w\in D_{r}. $$

We claim that  $|w|^{-\frac{2n-2}{p}}\tilde V(w)\in L^2(D_{r})$  for a.e. $\zeta\in S^{2n-1}$. In fact,    according to Lemma \ref{cp}, 
$$\int_{|z|< r} |  V(z)|^p dv_z = \frac{1}{2\pi} \int_{|\zeta| =1}  \int_{|w|< r}|w|^{2n-2}|V(w\zeta)|^pdv_wdS_\zeta = \frac{1}{2\pi} \int_{|\zeta| =1}  \int_{|w|< r} |\tilde V(w)|^pdv_wdS_\zeta.$$
In particular,  $\tilde V\in L^p(D_{r})$ for a.e. $\zeta\in S^{2n-1}$. 
By H\"older's inequality
\begin{equation*}\begin{split}
        \int_{|w|< r } \left||w|^{-\frac{2n-2}{p}}\tilde V(w)\right|^2dv_w&\le  \left( \int_{|w|< r }  |\tilde V(w)|^pdv_w\right)^{\frac{2}{p}}  \left(\int_{|w|< r } |w|^{-\frac{2n-2}{p}\frac{2p}{p-2}}dv_w\right)^{\frac{p-2}{p}}\\
        &=\left( \int_{|w|< r }  |\tilde V(w)|^pdv_w\right)^{\frac{2}{p}}  \left(\int_{|w|< r } |w|^{-\frac{4n-4}{p-2}}dv_w\right)^{\frac{p-2}{p}}.
\end{split}
\end{equation*}
Since $p>2n$, we have $\frac{4n-4}{p-2}<2 $ and thus $ \int_{|w|< r } |w|^{-\frac{4n-4}{p-2}}dv_w<\infty$. This, combined with  the fact that $\tilde V\in L^p(D_{r})$, gives the desired claim. Hence we  can make use of  Theorem \ref{pz} part 2) to obtain $v = 0$ on $D_{r}$ for a.e. $\zeta\in S^{2n-1}$. Thus $u=0$ on $B_{r}$. The weak unique continuation property in Theorem \ref{pz} part 1) further applies to give $u\equiv 0$ on $\Omega$. 
\end{proof}

\section{Unique continuation for potentials involving $\frac{1}{|z|}$ for $N =1$    }\label{se5}
Let $\Omega$ be a domain in $\mathbb C^n$ containing $0$.   Let $u: \Omega\rightarrow \mathbb C^N$ be an  $H^1_{loc}(\Omega)$ solution to the  inequality
\begin{equation}\label{z-1}
    |\bar\partial u|\le \frac{C}{|z|}|u|\ \ \text{a.e. on}\ \ \Omega, 
\end{equation}  
where $C$ is some positive constant. Note that the potential $\frac{C}{|z|}\notin L^{2n}_{loc}(\Omega)$. 
The goal of this section is to show  the unique continuation property    for \eqref{z-1} if the target dimension $N=1$ as stated below.

\begin{theorem}\label{33}
Let $\Omega$ be a   domain in $\mathbb C^n$ and $0\in \Omega$. Let $u: \Omega\rightarrow \mathbb C$ with $u\in  H^1_{loc}(\Omega)$, and satisfies $|\bar\partial u|\le \frac{C}{|z|}|u|$ a.e. on $\Omega$ for some constant $C>0$. If $u$ vanishes to infinite order  in the $L^2$ sense at $0$, then $u $ vanishes identically. 
\end{theorem}

To prove Theorem \ref{33},  we need a few preparation lemmas. 

\begin{lem}\label{22}
Let $u\in L^2$ near $0\in \mathbb C^n$, and vanishes to infinite order in the $L^2$ sense at $0$. Then for each $M>0$, $\frac{u}{|z|^M} \in L^2$ near $0$, and vanishes to infinite order in the $L^2$ sense at $0$. 
\end{lem}

\begin{proof}
For each $m\ge 1$ and $\epsilon>0$, by  the $L^2$ flatness of $u$ at $0$, there exists $\delta >0$ such that for $0<r\le \delta$, 
\begin{equation*}
    \int_{|z|<r}|u|^2dv_{z}\le \epsilon r^{m+2M}.
\end{equation*}
Then for $0<r\le \delta$,
\begin{equation*}
    \begin{split}
      \int_{|z|<r}\frac{|u|^2}{|z|^{2M}}dv_{z} =&     \sum_{j=1}^\infty \int_{\frac{r}{2^j}<|z|<\frac{r}{2^{j-1}}}\frac{|u|^2}{|z|^{2M}}dv_{z} 
      \le   \sum_{j=1}^\infty \frac{2^{2Mj}}{r^{2M}}\int_{\frac{r}{2^j}<|z|<\frac{r}{2^{j-1}}}|u|^2dv_{z}\\
      \le &\sum_{j=1}^\infty \frac{2^{2Mj}}{r^{2M}}\int_{|z|<\frac{r}{2^{j-1}}}|u|^2dv_{z} 
      \le  \epsilon \sum_{j=1}^\infty \frac{2^{2Mj}}{r^{2M}}\frac{r^{m+2M}}{2^{(m+2M)(j-1)}}\\
      =&\epsilon2^{2M}r^m\sum_{j=1}^\infty 2^{-m(j-1)}\le \epsilon2^{2M+1} r^m.
    \end{split}
\end{equation*}
In particular, $\frac{u}{|z|^M}\in L^2$ near $0$, and vanishes to infinite order in the $L^2$ sense at $0$.
\end{proof}
\medskip

\begin{lem}\label{11}
Let $\Omega$ be a   domain in $\mathbb C^n$ and $0\in \Omega$.  Assume that  $u\in L_{loc}^2(\Omega)$ and is holomorphic in $\Omega\setminus \{0\}$.  If   $u$ vanishes to infinite order  in the $L^2$ sense at $0$, then $u\equiv 0$. 
\end{lem}

\begin{proof}
Write  $u$ in terms of the Laurent expansion  $ u(z) =\sum_{\alpha\in \mathbb Z^n} a_\alpha z^\alpha$ near  $0$  for some constants $a_\alpha$, $\alpha =(\alpha_1, \ldots, \alpha_n) \in \mathbb Z^n$. 
Then for each $0<r<<1$,
 $$\int_{|z|<r}|u|^2 dv_{z} = \int_{|z|<r} \sum_{\alpha, \beta\in \mathbb Z^n} a_\alpha\overline {a}_\beta z^\alpha\overline {z}^\beta dv_{z} =    \sum_{\alpha\in \mathbb Z^n} |a_\alpha|^2 \int_{|z|<r}  |z_1|^{2\alpha_1}\cdots |z_n|^{2\alpha_n} dv_z.  $$
The $L^2$ integrability of $u$ near $0$ leads to $a_\alpha=0$ for any $\alpha =(\alpha_1, \ldots, \alpha_n)$ with some $\alpha_j<0$. Thus $u$ is holomorphic on $\Omega$. By Lemma \ref{fn}, we further see $u\equiv 0$.
\end{proof}
 \medskip


\begin{lem}\label{ele1} Let $\Omega$ be a bounded domain in $\mathbb C^n$ and  $0\in \Omega$. For any $\alpha, \beta>0$ with $\alpha +\beta =2n$,   there exists some constant $C>0$ such that 
$$\int_{\Omega}\frac{dv_{\zeta}}{|\zeta|^\alpha|\zeta-z|^\beta} \le C\left(1+\left|\ln |z|\right|\right), \ \ z\in \Omega.$$
\end{lem}

\begin{proof}
 Fix $z\in \Omega$ and let $t: = |z|$. Let $r_0>0$ be large such that $\Omega\subset B_{r_0}$. For all $\zeta\in \Omega\setminus B_{2t}$, $|\zeta-z|\ge  |\zeta| -t \ge  |\zeta| -\frac{1}{2}|\zeta| = \frac{1}{2}|\zeta|$. Hence
    \begin{equation*}
        \begin{split}
            \int_{\Omega\setminus B_{2t} } \frac{dv_{\zeta}}{|\zeta|^\alpha|\zeta-z|^\beta}\le 2^\beta\int_{\Omega\setminus B_{2t} } \frac{dv_{\zeta}}{|\zeta|^{\alpha+\beta}}\le     2^\beta \left(\int_{2t}^{r_0} \frac{dr}{ r  }\right)w_{2n-1}   \le C_1\left(1+\left|\ln t\right|\right) 
        \end{split}
    \end{equation*}
 for some constant $C_1>0$ independent of $z$.   On the other hand,  writing $z=tz_0$ with $z_0\in S^{2n-1}$ and applying a change of coordinates $\zeta = t\eta $, we have
    \begin{equation*}
        \begin{split}
             \int_{  B_{2t} } \frac{dv_{\zeta}}{|\zeta|^\alpha|\zeta-z|^\beta}=&   \frac{1}{t^{\alpha+\beta-2n}}\int_{  B_{2} } \frac{dv_{\eta}}{|\eta|^\alpha|\eta-z_0|^\beta} = \int_{  B_{2} } \frac{dv_{\eta}}{|\eta|^\alpha|\eta-z_0|^\beta} \\
             =& \int_{  B_{\frac{1}{2}} } \frac{dv_{\eta}}{|\eta|^\alpha|\eta-z_0|^\beta} +   \int_{  B_{2}\setminus B_{\frac{1}{2}} } \frac{dv_{\eta}}{|\eta|^\alpha|\eta-z_0|^\beta}\\
             \le & 2^\beta\int_{  B_{\frac{1}{2}} } \frac{dv_{\eta}}{|\eta|^\alpha} +  2^\alpha \int_{  B_{2}\setminus B_{\frac{1}{2}} } \frac{dv_{\eta}}{|\eta-z_0|^\beta} \le C_2
        \end{split}
    \end{equation*}
 for some constant $C_2>0$ independent of $z$.   Altogether, we get the desired inequality.
\end{proof}
\medskip

Let us begin with the  case when the source dimension $n=1$. In fact, we shall prove the unique continuation for a   larger class of  potentials,  which take  on a hybrid form involving  both   powers  of $\frac{1}{|z|}$ and  Lebesgue integrable functions. Note that none of these potentials  below  belongs to $L^2_{loc}$.

\begin{theorem}\label{333}
Let $\Omega$ be a   domain in $\mathbb C$ containing $0$ and $1 < \beta < \infty$. Suppose $u: \Omega\rightarrow \mathbb C$ with $u\in  H^1_{loc}(\Omega)$, and satisfies 
$$|\bar\partial u|\le \frac{ \tilde V}{|z|^\frac{\beta-1}{\beta}}|u|\ \ \  \text{a.e.  on}\ \ \Omega   $$
 for some $\tilde  V\in L^{2\beta}_{loc}(\Omega)$. If $u$ vanishes to infinite order in the $L^2$ sense at $0$, then $u $ vanishes identically. 
\end{theorem}

\begin{proof} 
     Firstly,  $u \in   L^p_{loc}(\Omega)$ for all $p<\infty$ by  Sobolev embedding theorem. Since  $\tilde V\in L^{2\beta}_{loc}(\Omega)$ with $2\beta>2$, we have $ \tilde  V |z|^\frac{1-\beta}{\beta}u\in L_{loc}^{p_0}(\Omega\setminus \{0\})$ for some $p_0>2$ by H\"older's inequality. Thus by the ellipticity Lemma \ref{el}, $u\in W^{1, p_0}_{loc}(\Omega\setminus \{0\})\subset C^0(\Omega\setminus \{0\})$, the space of    continuous functions on $\Omega\setminus \{0\}  $, as a consequence of Sobolev embedding theorem.

For any subdomain $\tilde \Omega\subset \subset \Omega$ containing $0$, set $S: =\{z\in \tilde \Omega\setminus \{0\}: u(z) = 0\}$ and let 
\begin{equation*}\label{lam}
     V : = \begin{cases}
 \frac{\bar\partial u}{u}, &\text{on}\ \  (\tilde \Omega\setminus\{0\})\setminus S; \\
0, &\text{on}\ \ \ S\cup\{0\}.
\end{cases} 
\end{equation*} Then  \begin{equation}\label{aa}
      \bar\partial u  = Vu\ \ \text{on}\ \   (\tilde\Omega\setminus\{0\} )\setminus S  
\end{equation} in the sense of distributions. Since $ |V|\le \tilde  V |z|^\frac{1-\beta}{\beta}\in L^q (\tilde \Omega )$ for all $ 1\le q<2$, letting 
$$f(z): = \frac{1}{\pi }\int_\Omega\frac{ V(\zeta)}{\zeta-z}dv_{\zeta},\ \ \ z\in \tilde\Omega,   $$
one has $$\bar\partial f = V \ \ \text{on} \ \ \Omega$$ in the sense of distributions. See, for instance, \cite{Ve}. 
 Moreover,   by H\"older's inequality and   Lemma \ref{ele1} with $n=1$,  
 $$ |f(z)|\le \| \tilde V\|_{L^{2\beta} (\tilde \Omega )}\left(\int_{\tilde \Omega}\frac{dv_{\zeta}}{|\zeta|^\frac{2\beta-2}{2\beta-1}|\zeta-z|^\frac{2\beta}{2\beta-1}}\right)^\frac{2\beta-1}{2\beta}\le   M\left(1+\left|\ln |z|\right|\right), \ \ z\in \tilde \Omega$$    for some constant $M>0$.
  Hence   there exists some $C>0$ such that   \begin{equation}\label{hh1}
   \left| e^{- f}\right| \le \frac{C}{|z|^M}\ \ \text{on}\ \ \tilde \Omega. 
\end{equation}

On the other hand, restricting  on ${\tilde \Omega}\setminus \{0\}$, we have  $ V\in L^{2\beta}_{loc} ({\tilde \Omega}\setminus \{0\} ) $. Hence $ f\in W_{loc}^{1, 2\beta} ({\tilde \Omega}\setminus \{0\} )$ by  Lemma \ref{el},  and further  $ f\in C^0 ({\tilde \Omega}\setminus \{0\} )$. Applying Proposition \ref{pde} to $ f$ on  ${\tilde \Omega}\setminus \{0\}$, we get
\begin{equation}\label{hh2}
    \bar\partial e^{- f} = -e^{- f}  V \ \ \text{on}\ \  {\tilde \Omega}\setminus \{0\}
\end{equation}
 in the sense of distributions. 

Let $h: =ue^{- f} $ on ${\tilde \Omega}\setminus \{0\}$. Then $h\in  C^0 ({\tilde \Omega}\setminus \{0\} )$. Repeating a similar argument as in the proof of Theorem \ref{main} to $h$ on $ ({\tilde \Omega}\setminus\{0\} )\setminus S$,  and using \eqref{aa} and \eqref{hh2}, one can further show that $\bar\partial h =0 $  on $ ({\tilde \Omega}\setminus\{0\} )\setminus S$. Noting that $f\in C^0 ({\tilde \Omega}\setminus \{0\} )$, we have $S=  h^{-1}(0)$ and so $\bar\partial h =0 $  on $ ({\tilde \Omega}\setminus\{0\} )\setminus h^{-1}(0) $. By Rado's theorem, $h$ is holomorphic on ${\tilde \Omega}\setminus \{0\}$. On the other hand, since $u$ vanishes to infinite order in the $L^2$ sense  at $0$, by \eqref{hh1} and Lemma \ref{22}, $h\in L^2_{loc} ({\tilde \Omega} )$ and   vanishes to infinite order in the $L^2$ sense  at $0$ as well. As a consequence of Lemma \ref{11}, $h = 0 $ and thus $u  = 0$ on $\tilde \Omega$.
\end{proof}
 \medskip

When $\beta=1$, the unique continuation holds due to Theorem \ref{pz} part 2). When $\beta = \infty$, by employing exactly the same argument as in the proof of Theorem \ref{333}, with    the index $ \frac{\beta-1}{\beta}$   replaced by  $1$ ($= \lim_{\beta\rightarrow \infty} \frac{\beta-1}{\beta}$), and $\tilde V$  by a positive constant $C$, one can obtain the following unique continuation property. As a result, it  resolves Theorem \ref{33} for $n=1$.

\begin{theorem}\label{3333}
Let $\Omega$ be a   domain in $\mathbb C$ containing $0$. Suppose $u: \Omega\rightarrow \mathbb C$ with $u\in  H^1_{loc}(\Omega)$, and satisfies  
$$  |\bar\partial u|\le  \frac{ C}{|z|}|u|\ \ \text{a.e. on}\ \ \Omega$$ for some constant $C>0$.  If $u$ vanishes to infinite order in the $L^2$ sense at $0$, then $u$ vanishes identically. 
\end{theorem}
\medskip

Next, we address the case where the source dimension $n\ge 1$. We first explore some direct applications of Theorem \ref{333} to the unique continuation problem for smooth solutions.

\begin{proof}[Proof of Corollary \ref{prn}: ]  Assume  that \eqref{eq4} fails, say for $z_0=0$.  Then   there exists some $V\in L^{2n}_{loc}(U)$ such that $|\bar\partial u|\le V|u|$ on $U$. When  $n=1$,  $u\equiv 0$ on $U$ due to Theorem \ref{pz} part 2). So we assume $n\ge 2$. Let $r$ be small such that $B_{r}\subset U$.  For each   $\zeta\in S^{2n-1}$, let $\tilde V(w): = |w|^{\frac{n-1}{n}}V(w\zeta)$ and $v(w): = u(w\zeta), w\in D_{r}$. Then $v$ vanishes to infinite order at $0$ in the $L^2$ sense  by Lemma \ref{fn}, and   satisfies
$$|\bar\partial v(w) | = |\zeta \cdot \bar\partial u(w\zeta)|\le V(w\zeta) |u(w\zeta)| =   |w|^{-\frac{n-1}{n}}\tilde V(w)|v(w)|,\ \ w\in D_{r}. $$ Moreover, for a.e. $\zeta\in S^{2n-1}$, $\tilde V\in L^{2n}_{loc}(D_{r})$ by Lemma \ref{cp}. 
 Theorem \ref{333} with $\beta =n$  applies to give    $v = 0$ on $D_{r}$ for a.e. $\zeta\in S^{2n-1}$. Hence  $u=0$ on $B_{r}$. The weak unique continuation property Theorem in \ref{pz} part 1)  further leads to  $u\equiv 0$ on $U$.  
 \end{proof}
\medskip

  Theorem \ref{333}  also allows us to recover a similar result as in Theorem \ref{main2} for smooth solutions  without imposing the $\bar\partial$-closedness assumption on the potential.

\begin{cor}\label{nod}
     Let $\Omega$ be a  domain in $\mathbb C^n$. Suppose $u: \Omega\rightarrow \mathbb C$ with $u\in C^\infty(\Omega)$,  and satisfies $ |\bar\partial u| \le V|u|$ a.e. on $\Omega$ for some    $V \in L_{loc}^{2n}(\Omega)$.  If $u$ vanishes to infinite order   at some $z_0\in \Omega$, then $u$ vanishes identically.
\end{cor}

\begin{proof}
     Let $r$ be small such that $B_{r}\subset  \Omega$. As in the proof of Corollary \ref{prn}, for each fixed $\zeta\in S^{2n-1}$, let $\tilde V(w): = |w|^{\frac{n-1}{n}}V(w\zeta)$ and $v(w): = u(w\zeta), w\in D_{r}$. Then  $v$ vanishes to infinite order at $0$ in the $L^2$ sense, and   satisfies
$$|\bar\partial v(w) |\le  |w|^{-\frac{n-1}{n}}\tilde V(w)|v(w)|,\ \ w\in D_{r}. $$
 For a.e. $\zeta\in S^{2n-1}$,  we apply   Theorem \ref{333}  to get $v = 0$ on $D_{r}$.   Hence  $u=0$ on $B_{r}$. The weak unique continuation property  further  gives $u\equiv 0$.    
\end{proof}
\medskip

We are now in a position to prove Theorem \ref{33} for $H^1_{loc}$ solutions. We shall use the following    special case of  \cite[Theorem A]{GR} concerning the zero set of solutions to \eqref{eqn2} with bounded potentials.

\begin{theorem}\cite{GR}\label{gr1}
  Let $\Omega$ be a domain in $\mathbb C^n$. Suppose  $u: \Omega\rightarrow \mathbb C$  with $u\in C^0(\Omega)$, and   satisfies  $  |\bar\partial u |  \le C|u|$ a.e. on $\Omega$ for some constant $C>0$. Then the zero set $u^{-1}(0) $ of $u$ is a complex analytic variety.  
  \end{theorem}

\begin{proof}[Proof of Theorem \ref{33}:]
 The   $n=1$ case was proved in Theorem \ref{3333}. So we assume $n\ge 2$. According to  Sobolev embedding theorem,    $u\in L^{q^*}_{loc}(\Omega )$ for  $q^*= \frac{2n}{n-1}$.   The ellipticity Lemma \ref{el} of $\bar\partial$ and the inequality $|\bar\partial u|\le \frac{C}{|z|}|u|\in L^{q^*}(\Omega\setminus \{0\})$ further ensures  $u\in W^{1,q^*}_{loc}(\Omega\setminus \{0\})$, and thus $u\in L^{q^{**}}$ with $q^{**}= \frac{2n}{n-2}$. Employing a boot-strap argument eventually gives $u\in W_{loc}^{1, q_0}(\Omega\setminus \{0\})$ for some $q_0>2n$. In particular, $u\in C^0(\Omega\setminus \{0\})$.

Set $S: =\{z\in \Omega\setminus \{0\}: u(z) = 0\}$. Let 
\begin{equation*}\label{lam1}
     V : = \begin{cases}
 \frac{\bar\partial u}{u}, &\text{on}\ \ \left(\Omega\setminus\{0\}\right)\setminus S; \\
0, &\text{on}\ \ \ S\cup\{0\}.
\end{cases} 
\end{equation*}
Then  \begin{equation}\label{off}
    \bar\partial u =  V u \ \ \text{on}\ \ \left(\Omega\setminus\{0\}\right)\setminus S 
\end{equation}   in the sense of distributions. One can also verify that   
\begin{equation}\label{off1}
    \bar\partial  V =0  \ \ \text{on}\ \   \left(\Omega\setminus\{0\}\right)\setminus S
\end{equation}  
in the sense of distributions. In fact, given a  subdomain $U\subset \subset \left(\Omega\setminus\{0\}\right)\setminus S $, since $u\in W_{loc}^{1, q_0}(\Omega\setminus \{0\}) \subset C^0(\Omega\setminus \{0\})$,  $|u|>c$ on $U$ for some constant $c>0$. Letting $\{u_j\}_{j=1}^\infty\in C^\infty(U)\rightarrow u$  in the $ W^{1, q_0}(U) $ norm. In particular,  $u_j\rightarrow u$ in the $C^0(U)$ norm. Thus by passing to a subsequence, we can assume $|u_j|> \frac{c}{2}$ on $U$.  Let $V_j: =  \frac{\bar\partial u_j}{u_j}$ on $U$. Clearly $ \bar\partial V_j= 0$ on $U$. Moreover, 
\begin{equation*}
    \begin{split}
    \|V_j - V \|_{L^1(U)} = & \left\|\frac{\bar\partial u_j}{u_j} - \frac{\bar\partial u}{u}\right\|_{L^1(U)} \le    \left\| \frac{1}{u_j}(\bar\partial u_j - \bar\partial u) \right\|_{L^1(U)}   +\left\| \left(\frac{1 }{u_j} - \frac{1}{u}\right)\bar\partial u  \right\|_{L^1(U)}\\
    \le& \frac{2}{c}  \left\|  \bar\partial u_j - \bar\partial u  \right\|_{L^1(U)}   +\frac{2}{c^2}\|u_j -u\|_{C^0(U)}  \left\|  \bar\partial u  \right\|_{L^1(U)}\rightarrow 0 
    \end{split}
\end{equation*}
as $j\rightarrow \infty$. In particular,  $ V_j\rightarrow V$ in the sense of distributions. Thus  the $\bar\partial$-closedness passes onto  $V$ on $ \left(\Omega\setminus\{0\}\right)\setminus S $.

Since $| V|\le \frac{C}{|z|}\in L^\infty_{loc}(\Omega\setminus \{0\}) $, according to Theorem \ref{gr1}, $S$ is a complex analytic variety in $ \Omega\setminus \{0\}$. If $S=\Omega\setminus \{0\}$, then we are done. Otherwise,  $S$ is  of complex dimension less than $n$.  Since $n\ge 2$, $V\in L^2_{loc}(\Omega)$.  Applying  a removable singularity result of Demailly \cite[Lemma 6.9]{De} to \eqref{off1},  \begin{equation}\label{aa7}
    \bar\partial  V =0  \ \ \text{on}\ \    \Omega
\end{equation}    in the sense of distributions. On the other hand, noting that     $u\in L_{loc}^2(\Omega)$ and $  V u \in L^1_{loc}(\Omega)$,  we can apply Demailly's result   to \eqref{off} and obtain      \begin{equation}\label{aa1}
      \bar\partial u  = Vu\ \ \text{on}\ \   \Omega  
\end{equation} 
in the sense of distributions.

Let $f_0, u_0$ and $v_0$ be as in \eqref{bm} with $ V$ defined above. Then $\bar\partial f_0 = V$ on $B_r $ in the sense of distributions  for some $r>0$. Restricting on ${B_r}\setminus\{0\} $, use the ellipticity Lemma \ref{el} to further obtain $ f_0\in W_{loc}^{1, p}({B_r}\setminus\{0\} )$ for all $p<\infty$. Hence we can apply Proposition \ref{pde} to $ f_0$ on  $B_r\setminus \{0\}$ and get
\begin{equation}\label{hh7}
    \bar\partial e^{- f_0} = -e^{- f_0}   V \ \ \text{on}\ \  B_r\setminus \{0\}
\end{equation}
 in the sense of distributions.   Let $$h: =ue^{- f_0}   \ \ \text{on} \ \ B_r\setminus\{0\}. $$   Repeating a similar argument as in the proof of Theorem \ref{main} to $h$ on $ B_r\setminus\{0\}  $,  and using \eqref{aa1} and \eqref{hh7}, one can further show that $$\bar\partial h =0 \ \ \text{  on}\ \   B_r\setminus\{0\}.$$   
Namely, $h$ is holomorphic on $B_r\setminus\{0\} $.

On the other hand, by the construction of $f_0$ in Lemma \ref{el} and the fact that $|V|\le \frac{C}{|z|}$ on $\Omega$, we have $v_0$ to be bounded on $B_\frac{r}{2} $. Moreover, apply  Lemma \ref{ele1} to $ u_0$ and get   $$ |f_0(z)|\le C\left(1 +\int_{B_r}\frac{ dv_{\zeta}}{  |\zeta||\zeta-z|^{2n-1}}\right) \le M\left(1+|\ln|z||\right), \ \ z\in B_\frac{r}{2}$$ for some constant $M>0$. Hence   there exists some $C_1>0$ such that   \begin{equation} \label{re1}
   \left| e^{-f_0}\right| \le \frac{C_1}{|z|^M}\ \ \text{on}\ \  B_\frac{r}{2}.
\end{equation}
 Since $u$ vanishes to infinite order in the $L^2$ sense  at $0$, by \eqref{re1} and Lemma \ref{22}, $h\in L^2_{loc}(B_r)$ and   vanishes to infinite order in the $L^2$ sense  at $0$ as well. As a consequence of Lemma \ref{11}, $h =0 $ on $B_r$ and thus $u (= e^ 
 {f_0} h) = 0$ on $B_r$. Since $ \frac{C}{|z|} \in L^\infty_{loc}(\Omega\setminus \{0\}) $, applying the weak unique continuation property of $|\bar\partial u|\le V|u|$ on $\Omega\setminus \{0\}$ with $L^\infty_{loc}$ potentials, one further gets $u\equiv 0$.
\end{proof}
\medskip

\begin{remark}
  It should be pointed  out that, although the statement of \cite[Theorem A]{GR} does not   explicitly mention it, the $\bar\partial$-closedness of $V$ as indicated in \eqref{aa7} near $S$ has already been established in its proof  towards the analyticity  of the zero set $S$. We opt to utilize the statement directly and subsequently employ  Demailly's result to demonstrate it for the convenience of readers. 
\end{remark}

\section{Unique continuation for potentials involving $\frac{1}{|z|}$ for $N \ge 2$    }\label{se7}
In this section, we   study the unique continuation for a $H_{loc}^1(\Omega)$ solution $u: \Omega\rightarrow \mathbb C^N$   to the  inequality 
\begin{equation*} 
    |\bar\partial u|\le \frac{C}{|z|}|u|\ \ \text{a.e. on}\ \ \Omega, 
\end{equation*}
when the target dimension $N\ge 2$.
As  stated in Theorem \ref{33},    the unique continuation property  holds true when   $N=1$  for any constant multiple $C>0$ in the potential. However, when   $N\ge 2$, this    property  no longer holds in general if $C$ is  large,  as indicated by an example   below  of the first author and  Wolff in \cite{PW98}.   See also \cite{AB94} by Alinhac and Baouendi for an alternative example. 

\begin{example}\label{ex2}
    Let   $v_0: \mathbb C\rightarrow \mathbb C$ be the  nontrivial smooth scalar function constructed in \cite{PW98} that vanishes to infinite order at $0$ and  satisfies $|\triangle v_0|\le \frac{C^\sharp}{|z|}|\nabla v_0|$  on $\mathbb C$ for some constant $C^\sharp>0 $. Letting  $u_0: = ( \partial \Re v_0, \partial \Im  v_0)$, then $u_0: \mathbb C\rightarrow \mathbb C^2$ is smooth,  vanishes to infinite order  at $0$, and  satisfies $|\bar\partial u_0|\le \frac{C^\sharp}{2|z|}|  u_0|$ on $\mathbb C$. 
\end{example}

In spite of Example \ref{ex2}, we shall prove  that the unique continuation property still holds if the constant multiple $C$ is small enough.

\begin{theorem}\label{mm1}
Let $\Omega$ be a  domain in $\mathbb C$ and $0\in \Omega$. Let  $u: \Omega\rightarrow \mathbb C^N$ with $u\in H_{loc}^1(\Omega)$, and  satisfy $|\bar\partial u|\le \frac{C}{|z|}|u| $ a.e. on $\Omega$ for some positive constant $C<\frac{1}{4 }$. If $u$ vanishes to infinite order in the $L^2$ sense  at  $0$, then $u $ vanishes identically.
\end{theorem}

In order to prove Theorem \ref{mm1}, we need to establish  a 
Carleman inequality for $\bar\partial$ (and its conjugate $\partial)$, making use of  a  Fourier analysis method, along with   the following  lemma.

\begin{lem}\label{gg}
Let $f: (-\infty, 0) \rightarrow \mathbb C^N$ with $f\in C_c^\infty((-\infty, 0))$. Then for any $\lambda, k\in \mathbb R$,
$$  \int_{-\infty}^0e^{-2\lambda t}\left|(\partial_t+k)f(t)\right|^2dt\ge (\lambda+k )^2\int_{-\infty}^0 e^{-2\lambda t}|f(t)|^2 dt. $$
\end{lem}

\begin{proof}
Letting $g(t):  = e^{-\lambda t}f(t), t\in ( -\infty, 0)$, then its derivative  $g_t = e^{-\lambda t}(f_t-\lambda f)$, and further
$$ e^{-\lambda t} (\partial_t+k)f   = e^{-\lambda t}(f_t+kf) = g_t +(\lambda +k)g = (\partial_t +\lambda +k)g. $$
Consequently, 
\begin{equation*}
    \begin{split}
        \int_{-\infty}^0e^{-2\lambda t}\left|(\partial_t+k)f(t)\right|^2dt =& \int_{-\infty}^0\left|(\partial_t +\lambda +k)g(t)\right|^2dt\\
        =&\int_{-\infty}^0 |g_t(t)|^2 dt+(\lambda +k)^2 \int_{-\infty}^0|g(t)|^2 dt + 2(\lambda+k)Re \int_{-\infty}^0\langle g_t(t),  g(t)\rangle dt\\
        \ge& (\lambda +k)^2\int_{-\infty}^0e^{-2\lambda t}|f(t)|^2 dt+ 2(\lambda+k)Re \int_{-\infty}^0\langle g_t(t),  g(t)\rangle dt.
    \end{split}
\end{equation*}
Note that since $g\in  C_c^\infty((-\infty, 0))$, 
$$ 0= \int_{-\infty}^0 \frac{d}{dt}(|g|^2) dt = 2Re \left(\int_{-\infty}^0\langle g_t, g\rangle dt\right). $$ We obtain the desired inequality.
\end{proof}
\medskip

\begin{pro}\label{hh3}
For any $u: \mathbb C\rightarrow \mathbb C^N$ with $u\in H^1(\mathbb C)$ and  supported outside a neighborhood of $0$, and for any $\lambda \in \mathbb Z+\left\{\frac{1}{2} \right\}$,  
\begin{equation}\label{ha}
   \int_{\mathbb C}\frac{|u(z)|^2}{|z|^{2\lambda+2}} dv_z  \le 16\int_{\mathbb C}\frac{|\partial u(z)|^2}{|z|^{2\lambda}} dv_z \end{equation}
  and
   \begin{equation*}\int_{\mathbb C}\frac{|u(z)|^2}{|z|^{2\lambda+2}} dv_z  \le 16\int_{\mathbb C}\frac{|\bar\partial u(z)|^2}{|z|^{2\lambda}} dv_z.
\end{equation*} 
\end{pro}

\begin{proof}
We shall only prove \eqref{ha} in terms of  $\partial u $, as $\bar\partial u=\overline{\partial \bar u}$. Since the proof involves derivatives on other variables as well, we use $u_z$ instead of $ \partial u$ to emphasize its derivative with respect to  $z$. 

First, we consider $u\in C_c^\infty(\mathbb C\setminus\{0\})$. Since the inequality is scaling-invariant, without loss of generality we  assume $u$ is supported inside the unit disc $ D_1$.  Let $v(t, \theta): = u(e^{t+i\theta}),$ $t\in (-\infty, 0), \theta \in (0, 2\pi)$. Write the Fourier series of $v$ as $$v(t, \theta) = \sum_{k\in \mathbb Z} v_k(t) e^{ik\theta}, $$ where    
$$v_k(t): = \frac{1}{2\pi}\int_0^{2\pi} v(t, \theta) e^{-ik\theta}d\theta \in C_c^\infty((-\infty, 0)). $$
According to the Parseval's identity,
\begin{equation}\label{pa}
    \int_0^{2\pi}|v(t, \theta)|^2d\theta = 2\pi \sum_{k\in \mathbb Z}|v_k(t)|^2.
\end{equation}   
Then  under the coordinate change $r= e^t$, we have
\begin{equation}\label{pp}
    \begin{split}
  \int_{D_1} \frac{|u(z)|^2}{|z|^{2\lambda+2}}dv_{z} = &\int_0^{2\pi}\int_0^1 r^{-2\lambda-1}|u(r, \theta)|^2drd\theta 
  =  \int_0^{2\pi}\int_{-\infty}^0 e^{(-2\lambda-1)t}|v(t, \theta)|^2e^{t}dtd\theta\\
  =& \int_{-\infty}^0\int_0^{2\pi} e^{-2\lambda t}|v(t, \theta)|^2d\theta dt.
    \end{split}
\end{equation}

On the other hand,  note that for $z = e^{t}e^{i\theta}$, one has $z \partial_z =  \frac{1}{2}\left(\partial_t -i\partial_\theta    \right) $. Thus  $e^{t+i\theta} v_z = zv_z = \frac{1}{2}\left( {\partial_t} -i{\partial_\theta}    \right) v= \frac{1}{2}\sum_{k\in \mathbb Z} (\partial_t+k)v_k(t) e^{ik\theta} $ and 
$$\int_0^{2\pi}|e^{t}v_z(t, \theta)|^2d\theta  = \frac{\pi}{2}\sum_{k\in\mathbb Z} \left|(\partial_t+k)v_k(t)\right|^2.  $$
Hence  
 \begin{equation*}
     \begin{split}
      \int_{D_1} \frac{|u_z(z)|^2}{|z|^{2\lambda}}dv_{z} = &\int_0^{2\pi}\int_0^1    r^{-2\lambda+1}|u_z(re^{i \theta})|^2drd\theta  
      = \int_0^{2\pi}\int_{-\infty}^0   e^{-2\lambda t +2t}|v_z(t, \theta)|^2dtd\theta\\
      =&\int_{-\infty}^0 e^{-2\lambda t}\int_0^{2\pi}|e^tv_z(t, \theta)|^2d\theta  dt 
      = \frac{\pi}{2} \sum_{k\in\mathbb Z}\int_{-\infty}^0 e^{-2\lambda t}\left| (\partial_t+k)v_k(t)  \right|^2dt.
     \end{split}
 \end{equation*}
Applying Lemma \ref{gg} to $v_k$ and making use of the fact that $(\lambda +k)^2\ge \frac{1}{4} $ whenever $\lambda \in \mathbb Z +\left\{\frac{1}{2} \right\}$ and $k\in \mathbb Z$, 
\begin{equation*}
     \begin{split}
      \int_{D_1} \frac{|u_z(z)|^2}{|z|^{2\lambda}}dv_{z} \ge &\frac{\pi}{2} \sum_{k\in\mathbb Z}(\lambda+k)^2 \int_{-\infty}^0 e^{-2\lambda t}\left| v_k(t)  \right|^2dt
      \ge  \frac{\pi}{8} \sum_{k\in\mathbb Z} \int_{-\infty}^0 e^{-2\lambda t}\left| v_k(t)  \right|^2dt \\
       = &\frac{1}{16} \int_{-\infty}^0 \int_0^{2\pi}e^{-2\lambda t}|v(t, \theta)|^2d\theta dt.
     \end{split}
 \end{equation*}
Here in the last line we also used \eqref{pa}. The   inequality \eqref{ha} for $u\in C_c^\infty(\mathbb C\setminus\{0\}) $ is proved by combining the above inequality with \eqref{pp}.

For general $u\in H^1(\mathbb C)$ in the proposition, let $r>0$ be small such that the support of $u$ is outside $D_r$. Pick  a family  $u_j \in C_c^\infty(\mathbb C\setminus D_r)\rightarrow u$ in $H^1(\mathbb C)$ norm. Then applying  \eqref{ha} to $u_j\in C_c^\infty(\mathbb C\setminus \{0\})$, we get
\begin{equation*}
    \begin{split}
   \left(\int_{\mathbb C}\frac{|u(z)|^2 }{|z|^{2\lambda+2}}dv_z \right)^\frac{1}{2}  \le&    \left(\int_{\mathbb C\setminus D_r}\frac{|u(z)-u_j(z)|^2 }{|z|^{2\lambda+2}}dv_z\right)^\frac{1}{2} + \left(\int_{\mathbb C}\frac{|u_j(z)|^2}{|z|^{ 2\lambda+2}} dv_z\right)^\frac{1}{2} \\
    \le&  {r^{-\lambda-1}}\left(\int_{\mathbb C}|u(z)-u_j (z)|^2dv_z\right)^\frac{1}{2}+   4\left(\int_{\mathbb C}  \frac{| (u_j)_z (z)|^2}{|z|^{2\lambda}}dv_z\right)^\frac{1}{2}.
    \end{split}
\end{equation*}
 Since
 \begin{equation*}
     \begin{split}
      \left(\int_{\mathbb C}  \frac{| (u_j)_z (z)|^2}{|z|^{2\lambda}}dv_z\right)^\frac{1}{2}\le & \left(\int_{\mathbb C\setminus D_r} \frac{ |  u_z(z)- (u_j)_z (z)|^2}{|z|^{2\lambda } } dv_z\right)^\frac{1}{2}  
    +   \left(\int_{\mathbb C} \frac{|  u_z(z) |^2}{ |z|^{2\lambda }} dv_z\right)^\frac{1}{2}  \\
    \le &    r^{-\lambda }  \left(\int_{\mathbb C} | u_z(z)-(u_j)_z (z)|^2dv_z\right)^\frac{1}{2}  
     +   \left(\int_{\mathbb C} \frac{|  u_z(z) |^2}{|z|^{2\lambda } } dv_z\right)^\frac{1}{2}, 
     \end{split}
 \end{equation*}
 one thus has
\begin{equation*}
\begin{split}
      \left(\int_{\mathbb C}\frac{|u(z)|^2 }{|z|^{ 2\lambda+2}}dv_z \right)^\frac{1}{2}  
      \le   \left( {r^{-\lambda-1}}+  4  r^{-\lambda } \right)\|u-u_j \|_{H^1(\mathbb C)}+  4\left(\int_{\mathbb C} \frac{|  u_z(z) |^2}{|z|^{2\lambda} }dv_z\right)^\frac{1}{2}.
\end{split}
\end{equation*}
Letting $j\rightarrow \infty$, we   have the desired inequality \eqref{ha} for $u\in H^1(\mathbb C)$ with support away from $0$.
\end{proof}
\medskip

By employing an induction process along with a similar  argument as in the proof to Proposition \ref{hh3}, one  can further get the following higher order edition.

\begin{cor}\label{hh4}
Let $k, l\in \mathbb Z^+$ with $k\le l$, and  $\lambda \in \mathbb Z+\left\{\frac{1}{2}\right\} $. For any $u: \mathbb C\rightarrow \mathbb C^N$ with $u\in H_{loc}^{l}(\mathbb C)$  and supported outside a neighborhood of $0$, and any 2-tuples $\alpha =(\alpha_1, \alpha_2), \beta=(\beta_1, \beta_2)$ with $ |\alpha|=k, |\beta|=l $ and $\alpha_j\le \beta_j,  j=1, 2$, there exists a constant $C$ dependent only on $l$ such that
$$  \int_{\mathbb C}\frac{|\partial^{\alpha_1}\bar\partial^{\alpha_2}u(z)|^2 }{|z|^{2\lambda+2(l-k)}}dv_z  \le C\int_{\mathbb C}\frac{|\partial^{\beta_1}\bar\partial^{\beta_2}u(z)|^2}{|z|^{2\lambda}} dv_z.$$
\end{cor}
\medskip

\begin{proof}[Proof of Theorem \ref{mm1}:]
Let $r>0$ be small  such that $D_{2r}\subset\subset \Omega$. Choose $\eta\in C_c^\infty(\mathbb C)$ with  $\eta =1$ on $D_r$,  $0\le \eta\le 1$ and $|\nabla \eta|\le \frac{2}{r}$ on $D_{2r}\setminus  D_r$, and $\eta =0$  outside $D_{2r}$. Let $\psi\in  C^\infty(\mathbb C)$ be  such that $\psi =0$ in $D_1$, $  0\le \psi\le 1$ and $|\nabla \psi|\le 2$ on $D_{2}\setminus {D_1}$, and $\psi =1$ outside $D_2$.  For each $k\ge \frac{4}{r}$ (thus $ \frac{2}{k}\le\frac{r}{2}$),  let $\psi_k  = \psi(k\cdot)$ and $u_k = \psi_k\eta u$. Then $u_k\in H^1(\mathbb C )$ with support outside $D_\frac{1}{k}$. 

Since $C<\frac{1}{4}$, one can choose  $\epsilon_0>0$ with \begin{equation}\label{sc}
    16(1+2\epsilon_0)C^2<1.
\end{equation} 
Making use of the following elementary inequality  \begin{equation*} 
     (a+b+c)^2\le (1+2\epsilon)a^2+ (2+\epsilon^{-1}) b^2+ (2+\epsilon^{-1})c^2,\ \ \text{for all}\ \ a, b, c\in \mathbb R, \epsilon>0, 
 \end{equation*} together with Proposition \ref{hh3}   and the inequality \eqref{z-1}, we have for each $\lambda \in \mathbb Z+\left\{\frac{1}{2} \right\}$,  
\begin{equation*}
    \begin{split}
     \int_{D_{2r}}\frac{|u_k(z)|^2}{|z|^{2\lambda}}dv_z  \le& 16\int_{D_{2r}}\frac{|\bar\partial u_k(z)|^2}{|z|^{2\lambda-2}}dv_z\\
       \le & 16(1+2\epsilon_0)\int_{D_{2r}}\frac{|\psi_k(z) \eta(z) |^2 | \bar\partial u(z)|^2}{|z|^{2\lambda-2}}dv_z  + 16(2+\epsilon_0^{-1})\int_{D_r}\frac{|\bar\partial \psi_k(z)|^2 |  u(z)|^2}{|z|^{2\lambda-2}}dv_z\\  &+   16(2+\epsilon_0^{-1}) \int_{D_{2r}\setminus D_r}\frac{|\bar\partial \eta(z)  |^2 |  u(z)|^2}{|z|^{2\lambda-2}}dv_z \\
        \le &  16(1+2\epsilon_0)C^2\int_{D_{2r}}\frac{ |u_k(z)|^2}{|z|^{2\lambda}}dv_z  +  16(2+\epsilon_0^{-1})\int_{D_r}\frac{|\bar\partial \psi_k(z)|^2 |  u(z)|^2}{|z|^{2\lambda-2}}dv_z \\
        &+  16(2+\epsilon_0^{-1})\int_{D_{2r}\setminus D_r}\frac{|\bar\partial \eta(z)  |^2 |  u(z)|^2}{|z|^{2\lambda-2}}dv_z.
    \end{split}
\end{equation*}
 Noting that \eqref{sc} holds,   one can subtract $16(1+2\epsilon_0)C^2\int_{D_{2r}}\frac{ |u_k(z)|^2}{|z|^{2\lambda}}dv_z$ from both sides and get
\begin{equation}\label{hp}
    \int_{D_{2r}}\frac{|u_k(z)|^2}{|z|^{2\lambda}}dv_z  \le C_0\left( \int_{D_r}\frac{|\nabla \psi_k(z)|^2 |  u(z)|^2}{|z|^{2\lambda-2}}dv_z+\int_{D_{2r}\setminus D_r}\frac{|\nabla \eta(z)  |^2 |  u(z)|^2}{|z|^{2\lambda-2}}dv_z\right), 
\end{equation} 
where $$C_0: = \frac{16(2+\epsilon_0^{-1})}{1-16(1+2\epsilon_0)C^2}>0.$$

Next, we show that \begin{equation}\label{fl}
    \lim_{k\rightarrow \infty}  \int_{D_r}\frac{|\nabla \psi_k(z)|^2 |  u(z)|^2}{|z|^{2\lambda-2}}dv_z  = 0.
\end{equation} Indeed, since $\nabla \psi_k$ is only supported on $D_{\frac{2}{k}}\setminus D_{\frac{1}{k}}$,
$$ \int_{D_r}\frac{|\nabla \psi_k(z)|^2 |  u(z)|^2}{|z|^{2\lambda-2}}dv_z \le \int_{\frac{1}{k}<|z|<\frac{2}{k}} \frac{|\nabla \psi_k(z)|^2|u(z)|^2}{|z|^{2\lambda-2} } dv_{z}\le k^{2\lambda}   \int_{|z|<\frac{2}{k}}|u(z)|^2dv_z\rightarrow 0$$
 as $k \rightarrow \infty$, as a consequence of the flatness of $u$  at $0$ in the $L^2$ sense. 

Letting $k\rightarrow \infty$ in \eqref{hp}, and making use of \eqref{fl} and Fatou's Lemma, we obtain  that 
\begin{equation*}
    \int_{D_{2r}}\frac{| \eta(z)u(z)|^2}{|z|^{2\lambda}}dv_z  \le C_0\int_{D_{2r}\setminus D_r}\frac{|\nabla \eta(z)  |^2 |  u(z)|^2}{|z|^{2\lambda-2}}dv_z. 
\end{equation*}
  Since
  $$  \int_{D_{2r}}\frac{|\eta(z)u(z)|^2}{|z|^{2\lambda}}dv_z \ge  \int_{D_{\frac{r}{2}}}\frac{|u(z)|^2}{|z|^{2\lambda}}dv_z \ge\left(\frac{2}{r}\right)^{2\lambda}\int_{D_{\frac{r}{2}}}| u(z)|^2dv_z $$
and $$\int_{D_{2r}\setminus D_r}\frac{|\nabla \eta(z)  |^2 |  u(z)|^2}{|z|^{2\lambda-2}}dv_z \le \frac{1}{r^{2\lambda-2}}\int_{D_{2r}\setminus D_r}|\nabla \eta (z) |^2 |  u(z)|^2dv_z,$$
we have
$$\int_{D_{\frac{r}{2}}}| u(z)|^2dv_z \le   \frac{C_0r^2}{ 2^{2\lambda}}\int_{D_{2r}\setminus D_r} |\nabla \eta(z)  |^2 |  u(z)|^2dv_z. $$
Letting $\lambda\rightarrow \infty$, we see $u= 0$ on $D_\frac{r}{2}$. Finally, 
apply the   unique continuation property  Theorem \ref{pz}  part 1) to get $u\equiv 0$.
\end{proof}
\medskip

\begin{proof}[Proof of Theorem \ref{main4}:] Let $r$ be small such that $B_{r}\subset \subset \Omega$.  For each fixed $\zeta\in S^{2n-1}$, let   $v(w): = u(w\zeta), w\in D_{r}$. Then  $v$ vanishes to infinite order  in the $L^2$ sense at $0$ and   satisfies
$$|\bar\partial v(w) |= |\zeta \cdot \bar\partial u(w\zeta)|\le \frac{C}{|w|} |u(w\zeta)| =  \frac{C}{|w|} |v(w)|,\ \ w\in D_{r}. $$
 For a.e. $\zeta\in S^{2n-1}$,  we apply   Theorem \ref{33} when  $N=1$, or Theorem \ref{mm1} when $N\ge 2$ and $C< \frac{1}{4}$,  to get $v = 0$ on $D_{r}$.   Hence  $u=0$ on $B_{r}$ in either case. The weak unique continuation property    further applies to give $u\equiv 0$.    
\end{proof}

 \section{Proof of Theorem \ref{main6}}\label{se6}
In this section, we prove Theorem \ref{main6} -- the unique continuation property for $ |\bar\partial u| \le V|u|$   on $\Omega\subset \mathbb C^2$, with the target dimension $N\ge 1$, and  $V\in L^4_{loc}$.  As already seen in Section \ref{s4}, its proof   can be reduced to that of the following theorem on the complex plane.   

\begin{theorem}\label{main7}
      Let $\Omega$ be a  domain in $\mathbb C$ and $0\in \Omega$. Suppose   $u: \Omega\rightarrow \mathbb C^N$  with $u\in H^1_{loc}(\Omega)$, and    satisfies \begin{equation}\label{qq}
          |\bar\partial u| \le {|z|^{-\frac{1}{2}}} V |u| \ \ \text{a.e. on}\ \ \Omega 
      \end{equation}    for some    $V \in L_{loc}^{4}(\Omega)$.  If $u$ vanishes to infinite order in the $L^2$ sense at $0$, then $u$ vanishes identically.
\end{theorem}
 \medskip

Note that the $N =1$ case in  Theorem \ref{main7} is a special case that has been proved in Theorem \ref{333}. On the other hand, since ${|z|^{-\frac{1}{2}}} V\notin L^2_{loc}(\Omega) $ given a general $V \in L_{loc}^{4}(\Omega)$,  Theorem \ref{pz} does not apply.  

The key element in proving Theorem \ref{main7} involves an idea in \cite{PZ} that utilizes the Cauchy integral, coupled with the technique employed in establishing Theorem \ref{mm1}.  To begin with, let us recall a representation formula for  $u\in H^1(\mathbb C)$ with compact support in terms of the Cauchy kernel:  
\begin{equation}\label{cg}
    u(z) = \frac{1}{\pi}\int_{\mathbb C}\frac{\bar\partial u(\zeta)}{z -\zeta }dv_{\zeta}, \ \ \text{a.e.}\ z\in \mathbb C.
\end{equation}
See, for instance, \cite[Lemma 3.1]{PZ}. Denote by $\|f\|_{L^2_V(\Omega)}$ the weighted $L^2(\Omega)$ norm of a function $f$ on $\Omega\subset \mathbb C$ with respect to a weight $V>0$, with
 $$ \|f\|_{L^2_V(\Omega)}: =\left(\int_{\Omega} |f(z)|^2V(z)dv_z\right)^\frac12.  $$
It was  proved in \cite[Theorem 2.2]{PZ}  that,  given a positive function $V\in L^2(\mathbb C)$,  the Riesz potential  
$$I_{1}f= \int_{\mathbb C} \frac{f(\zeta)}{|\zeta -\cdot |}dv_\zeta $$ is a bounded operator   from $L^2_{V^{-1}} (\mathbb C)$   to $L^2_{V}(\mathbb C)$. More precisely,  there exists a universal constant $C_0$ such that for any $f\in L_{V^{-1}}^2(\mathbb C)$,  \begin{equation}\label{whl}
    \|I_1f\|_{L_V^2(\mathbb C)}\le C_0 \|V\|_{L^2(\mathbb C)} \|f\|_{L_{V^{-1}}^2(\mathbb C)}.
\end{equation}
 \medskip

\begin{proof}[Proof of Theorem \ref{main7}: ]
    Fix an $r>0$  small  such that $D_{2r}\subset \subset\Omega$,   and 
\begin{equation*}
    \left\|V\chi_{D_{r}}+\frac{r}{1+|z|^2}\right\|^4_{L^4(\mathbb C)}\le \frac{\pi^2}{32C^2_0}, 
\end{equation*} where $C_0$ is the universal constant in \eqref{whl}, and $\chi_{D_{r}}$ is the characteristic function for  ${D_{r}}$.  Replacing $V $ by $V \chi_{D_r}+\frac{r}{1+|z|^2}$, we have  $\eqref{qq}$  holds on $D_{2r}$ with  $V \in L^4(\mathbb C)$,  
    \begin{equation}\label{vb}
    V>0 \ \ \text{on}\ \ \mathbb C;   \ \  V\ge C_r \ \  \text{on}\ \ D_{2r}
    \end{equation} 
    for  some  $C_r>0$ dependent only on $r$, and
    \begin{equation}\label{vs}
    \|V\|^4_{L^4(\mathbb C)}\le \frac{\pi^2}{32C^2_0}.
    \end{equation}
    We shall show that $u = 0$ on $D_\frac{r}{2}$.  
    
Let $\eta$ and $\psi_k$ be as defined in the proof of Theorem \ref{mm1}. 
Then  $u_k: = \psi_k \eta u \in H^1(\mathbb C )$ and is supported inside $D_{2r}\setminus D_{\frac{1}{k}}$. So is $\frac{u_k}{z^m} $  for each $m\in \mathbb Z^+$.  
 Applying \eqref{cg} to $\frac{u_k}{z^m} $,  we obtain     
\begin{equation*}
    \begin{split}
         \frac{|u_k(z)|^2}{|z|^{2m}} =&\frac{1}{\pi^2} \left |\int_{\mathbb C}\frac{\bar\partial  u_k(\zeta) }{(z -\zeta)\zeta^m }dv_\zeta \right |^2, \ \ z\in D_{2r},
    \end{split}
\end{equation*}
and   with $\tilde V: = V^2$, one has
\begin{equation*}
    \begin{split}
         \int_{D_{2r}}\frac{  |u_k(z)|^2}{|z|^{2m}}\tilde V(z)dv_z \le &\frac{1}{\pi^2} \int_{D_{2r}}\left (\int_{\mathbb C} \frac{1}{|z-\zeta|} \frac{|\bar\partial   u_k(\zeta) |}{|\zeta|^m}dv_{\zeta} \right )^2\tilde V(z)dv_z 
        \le   \frac{1}{\pi^2}\left\|I_1\left(\frac{|\bar\partial   u_k |}{|\cdot|^{m}} \right)\right\|^2_{L^2_{\tilde V}(\mathbb C)}.
    \end{split}
\end{equation*}
Make use of   \eqref{whl} with respect to the weight $\tilde V$  to further infer
\begin{equation}\label{mm}
    \begin{split}
       & \int_{D_{2r}}\frac{  |u_k(z)|^2}{|z|^{2m}}\tilde V(z)dv_z\\
               \le &\frac{C^2_0}{\pi^2}\|\tilde V\|^{2 }_{L^{2}(\mathbb C)}\int_{\mathbb C} \frac{|\bar\partial \left(\psi_k(z)\eta(z)u(z)\right)|^2}{|z|^{2m}\tilde V(z)}  dv_{z}  \\
        \le &\frac{C^2_0}{\pi^2}\|V\|^{4 }_{L^{4}(\mathbb C)}\left(\int_{D_{2r}} \frac{|\bar\partial \psi_k(z)|^2|u(z)|^2}{|z|^{2m}\tilde V(z) } dv_{z} + \int_{D_r} \frac{|\psi_k(z)|^2|\bar\partial u(z) |^2}{|z|^{2m}\tilde V(z) } dv_{z} + \int_{D_{2r}\setminus D_r} \frac{|\bar\partial \left(\eta(z)u(z)\right) |^2}{|z|^{2m} \tilde V(z)} dv_{z}\right)\\
        =&: A+B+C.
    \end{split}
\end{equation}

Note that for $B$, by  the inequalities \eqref{qq}    and \eqref{vs}
\begin{equation*}
   B\le \frac{C^2_0}{\pi^2}\|V\|^{4 }_{L^{4}(\mathbb C)}\int_{D_r} \frac{|\psi_k(z)|^2| u(z) |^2}{|z|^{2m+1}}  dv_{z} \le  \frac{1}{32}\int_{D_{2r}} \frac{| u_k(z) |^2}{|z|^{2m+1}} dv_{z}.
\end{equation*}
 Thus we apply Theorem \ref{hh3} with $\lambda = m-\frac{1}{2}$ to have
\begin{equation*}
\begin{split}
      B\le  &\frac{1}{2}\int_{D_{2r}} \frac{| \bar\partial u_k(z) |^2}{|z|^{2m-1}} dv_{z}\\
                \le & \frac{1}{2}\int_{D_{2r}}\frac{|\psi_k(z) \eta(z) |^2 | \bar\partial u(z)|^2}{|z|^{2m-1}}dv_z  + \frac{1}{2}\int_{D_r}\frac{|\bar\partial \psi_k(z)|^2 |  u(z)|^2}{|z|^{2m-1}}dv_z +\frac{1}{2} \int_{D_{2r}\setminus D_r}\frac{|\bar\partial \eta(z)  |^2 |  u(z)|^2}{|z|^{2m-1}}dv_z \\
        \le &  \frac{1}{2}\int_{D_{2r}}\frac{ |u_k(z)|^2}{|z|^{2m}}\tilde V(z)dv_z  +  \frac{1}{2}\int_{D_r}\frac{|\bar\partial \psi_k(z)|^2 |  u(z)|^2}{|z|^{2m-1}}dv_z 
         +  \frac{1}{2}\int_{D_{2r}\setminus D_r}\frac{|\bar\partial \eta(z) |^2 |  u(z)|^2}{|z|^{2m-1}}dv_z\\
          =&: I_1+I_2+I_3.
\end{split}
\end{equation*}
Here we used \eqref{qq} in the third inequality. Combining \eqref{mm} with the above, 
\begin{equation*}
     \int_{D_{2r}}\frac{  |u_k(z)|^2}{|z|^{2m}}\tilde V(z)dv_z\le A +C + I_1+ I_2 +I_3.  
\end{equation*}
One  further subtracts $I_1$ from both sides to get
\begin{equation}\label{ee}
    \int_{D_{2r}}\frac{  |u_k(z)|^2}{|z|^{2m}}\tilde V(z)dv_z\le 2A +2C + 2I_2 +2I_3.   
\end{equation}
Similarly as in the proof to \eqref{fl} along with the fact that      $\tilde V\ge C^2_r$ on $ D_{2r}$,  one has $$\lim_{k\rightarrow \infty} A =\lim_{k\rightarrow \infty} I_2= 0.$$
Together, after passing $k\rightarrow \infty$ and using Fatou's Lemma in \eqref{ee}, we obtain
\begin{equation}\label{bb}
    \begin{split}
     \int_{D_{2r}}\frac{ |\eta(z)|^2 |u(z)|^2}{|z|^{2m}}\tilde V(z)dv_z\le  \int_{D_{2r}\setminus D_r} \frac{|\bar\partial \left(\eta(z)u(z)\right) |^2}{|z|^{2m} \tilde V(z)} dv_{z}  +   \int_{D_{2r}\setminus D_r}\frac{|\bar\partial \eta(z)  |^2 |  u(z)|^2}{|z|^{2m-1}}dv_z.
    \end{split}
\end{equation}

Now  multiply two sides of \eqref{bb} by $r^{2m}$. On the left hand side,     $$  \int_{D_{2r}}\frac{ r^{2m}  |\eta(z)|^2|u(z)|^2}{|z|^{2m}}\tilde V(z)dv_z\ge  \int_{D_{\frac{r}{2}}}\frac{r^{2m} }{|z|^{2m}}|u(z)|^2\tilde V(z)dv_z \ge  2^{2m}\int_{D_\frac{r}{2}}|u(z)|^2\tilde V(z)dv_z.$$
On the right hand side, using the fact that       $\tilde V\ge C^2_r$ on $ D_{2r}$ again,  
 \begin{equation*} 
 \begin{split}
     &\int_{D_{2r}\setminus D_r} \frac{r^{2m}|\bar\partial \left(\eta(z)u(z)\right) |^2}{|z|^{2m}\tilde V(z)} dv_{z}  +   \int_{D_{2r}\setminus D_r}\frac{r^{2m}|\bar\partial \eta(z)  |^2 |  u(z)|^2}{|z|^{2m-1}}dv_z\\
      \le &\frac{1}{C_r^2}\int_{D_{2r}\setminus D_r}  |\bar\partial \left(\eta(z)u(z)\right) |^2  dv_{z} +  r  \int_{D_{2r}\setminus D_r}|\nabla \eta(z)  |^2 |  u(z)|^2dv_z\le \tilde C_r\|u\|^2_{H^1(D_{2r})}, 
 \end{split}
    \end{equation*}
    for some   $\tilde C_r$ dependent only on $r$. Thus 
 
 $$  2^{2m}\int_{D_\frac{r}{2}}|u(z)|^2\tilde V(z)dv_z\le   \tilde C_r\|u\|^2_{H^1(D_{2r})}. $$
Letting $m\rightarrow \infty$ and making use of the positivity of $\tilde V$ on $D_{\frac{r}{2}}$, we have $u= 0$  on $D_{\frac{r}{2}}$. The  proof  is  thus complete as a consequence of the weak unique continuation property.  
\end{proof}
\medskip

\begin{proof}[Proof of Theorem \ref{main6}: ] As in the proof to Theorem \ref{main3} yet with $n=2$ and $p=4$, let $z_0=0$ and $r>0$ be small such that  $V\in L^{4}(B_{r})$.  For each fixed $\zeta\in S^{3}$, let $\tilde V(w): = |w|^{\frac{1}{2}}V(w\zeta)$ and $v(w): = u(w\zeta), w\in D_{r}$. Then  $v$ vanishes to infinite order at $0$ in the $L^2$ sense. Moreover, $v$ satisfies
$$|\bar\partial v(w) | \le   |w|^{-\frac{1}{2}}\tilde V(w)|v(w)|,\ \ w\in D_{r}. $$
Note that  for a.e. $\zeta\in S^{3}$, $\tilde V\in L^{4}(D_{r})$ by Lemma \ref{cp}. According to Theorem \ref{main7},    $v = 0$ on $D_{r}$ for a.e. $\zeta\in S^{3}$. Hence  $u=0$ on $B_{r}$. Apply the weak unique continuation property to get $u\equiv 0$.
 \end{proof}
\medskip

\begin{remark}\label{re}
    The following two questions still remain  open. In particular,  with an approach similar as in the proof to Theorem \ref{main6}, the resolution of Question 1  can be converted to that of Question 2. 
    \medskip
    
 \noindent\textbf{1.  } 
     Let $\Omega$ be a  domain in $\mathbb C^n, n\ge 3$ and $N\ge 2$. Suppose   $u: \Omega\rightarrow \mathbb C^N$ is smooth on $\Omega$ and   satisfies $ |\bar\partial u| \le V|u|$ a.e. on $\Omega$ for some    $V \in L_{loc}^{2n}(\Omega)$. If $u$ vanishes to infinite order at some $z_0\in \Omega$, does $u$ vanish  identically?  
        \medskip
        
     \noindent\textbf{2. }
     Let $\Omega$ be a  domain in $\mathbb C$ containing $0$, and  $n, N\in \mathbb Z^+$  with $n\ge 3, N \ge 2$. Suppose   $u: \Omega\rightarrow \mathbb C^N$ is smooth on $\Omega$ and    satisfies $ |\bar\partial u| \le {|z|^{-\frac{n-1}{n}}} V |u|$ a.e. on $\Omega$ for some    $V \in L_{loc}^{2n}(\Omega)$.  If $u$ vanishes to infinite order at $0\in \Omega$, does  $u$ vanish  identically?

\end{remark}

\bibliographystyle{alphaspecial}

\fontsize{11}{11}\selectfont

\vspace{0.7cm}
\noindent pan@pfw.edu,

\vspace{0.2 cm}

\noindent Department of Mathematical Sciences, Purdue University Fort Wayne, Fort Wayne, IN 46805-1499, USA.\\

\noindent zhangyu@pfw.edu,

\vspace{0.2 cm}

\noindent Department of Mathematical Sciences, Purdue University Fort Wayne, Fort Wayne, IN 46805-1499, USA.\\
\end{document}